\documentclass[10pt]{article}

\usepackage{graphicx}
\usepackage[small]{caption}
\usepackage{amsthm}
\usepackage{amsmath}
\usepackage{amsfonts}
\usepackage{amssymb}
\usepackage{mathrsfs}
\usepackage[usenames]{color}
\usepackage[colorlinks,citecolor=black,filecolor=black,linkcolor=black,urlcolor=black]{hyperref}
\usepackage{mathpazo}	

\renewcommand{\baselinestretch}{1.1}

\theoremstyle{plain}
\newtheorem{theorem}{Theorem}
\newtheorem{fact}[theorem]{Known Result}
\newtheorem{lemma}[theorem]{Lemma}
\newtheorem{proposition}[theorem]{Proposition}
\newtheorem{corollary}[theorem]{Corollary}

\newcounter{myBackupProp}

\theoremstyle{definition}
\newtheorem{definition}[theorem]{Definition}
\newtheorem{example}[theorem]{Example}

\newtheorem{remark}[theorem]{Remark}
\newtheorem{question}[theorem]{Question}

\def\Z{\mathbf{Z}}					

\DeclareMathOperator{\HH}{H}		
\DeclareMathOperator{\CH}{C}		
\DeclareMathOperator{\im}{Image}	
\DeclareMathOperator{\supp}{supp}	
\DeclareMathOperator{\Span}{span}	
\DeclareMathOperator{\bs}{bs}		
\DeclareMathOperator{\inte}{int}		
\DeclareMathOperator{\link}{Lk}		

\DeclareMathOperator{\MCG}{Mod}	
\DeclareMathOperator{\Homeo}{Homeo} 	
\def\CC{\mathscr{C}}				
\def\AC{\mathscr{A}}				

\DeclareMathOperator{\St}{St}		
\def\surface{\Sigma}				

\title{Finite rigid sets and homologically non-trivial spheres\\ in the curve complex of a surface}

\author{Joan Birman\footnote{The first author acknowledges partial support from the Simons Foundation, under Collaborative Research Award \# 245711}, Nathan Broaddus\footnote{The second author is partially supported by the National Science Foundation under grant DMS-1007059} and William Menasco}

\date{June 9, 2014}

\begin{document}
\maketitle

\begin{abstract}  \noindent  
 Aramayona and Leininger  have provided  a ``finite rigid subset''  $\mathfrak{X}(\surface)$ of the curve complex $\CC(\surface)$ of a surface $\surface = \surface^n_g$, characterized by the fact that any simplicial injection $\mathfrak{X}(\surface) \to \CC(\surface)$ is induced by a unique element of the mapping class group $\MCG(\surface)$.   
 In this paper we prove that, in the case of the sphere with $n\geq 5$ marked points,  the reduced homology class of the finite rigid set of Aramayona and Leininger is a $\MCG(\surface)$-module generator for the reduced homology of the curve complex $\CC(\surface)$, answering in the affirmative a question posed in \cite{AL}.  
 For the surface $\surface = \surface_g^n$ with $g\geq 3$ and $n\in \{0,1\}$ we find that the finite rigid set $\mathfrak{X}(\surface)$ of  Aramayona and Leininger contains a proper subcomplex $X(\surface)$ whose reduced homology class is a $\MCG(\surface)$-module generator for the reduced homology of $\CC(\surface)$  but which is not itself rigid.      
\end{abstract}


\section{Introduction and Statement of Results}
\label{sec:intro}

Let $\surface_g^n$ be the surface with genus $g \geq 0$ and $n \geq 0$ marked points.  The {\em mapping class group} $\MCG(\surface_g^n)$ is the group
\[
\MCG(\surface_g^n) = \Homeo^+(\surface_g^n) \big/ \Homeo^0(\surface_g^n)
\]
where $\Homeo^+(\surface_g^n) $ is the group of orientation preserving self-homeomorphisms of  the surface $\surface_g^n$ which permute the marked points and $\Homeo^0(\surface_g^n)$ is the path component of $\Homeo^+(\surface_g^n)$ containing the identity map.

Let $V \subset \surface_g^n$ be the set of $n$ marked points.  An {\em essential curve} $\gamma$ in $\surface_g^n$ is a simple closed curve in $\surface_g^n  -  V$ such that any disk in  $\surface_g^n$ bounded by $\gamma$ must contain at least two marked points.  
The isotopy class of $\gamma$ is its orbit under the action of $\Homeo^0(\surface_g^n)$.  A {\em curve system} is a non-empty set of isotopy classes of essential curves which have pairwise disjoint representatives and a {\em sub curve system} of a curve system is any non-empty subset of a curve system. 
The {\em curve complex} $\CC = \CC(\surface_g^n)$ is the simplicial complex whose vertices correspond to isotopy classes of essential curves in $\surface_g^n$ and whose $k$-simplices correspond to curve systems with $k+1$ curves.  
Henceforth we will abuse notation by not distinguishing between essential curves and their isotopy classes. 
In the cases of the sporadic surfaces of low complexity, $\surface_0^4$, $\surface_1^0$ and $\surface_1^1$, it is customary to include simplices for sets of essential curves with pairwise minimal intersection, but for group cohomological reasons we will not do this here.  
Hence for our purposes $\CC(\surface_0^4)$, $\CC(\surface_1^0)$ and $\CC(\surface_1^1)$ will be  disconnected complexes of dimension $0$.

The curve complex plays dual geometric and homological roles for the structure of the mapping class group.  
On the one hand we have results in the vein of Ivanov \cite{Ivanov}, which retrieve the mapping class group as the set of simplicial isomorphisms of the curve complex and other related simplicial complexes.  
On the other hand we have the work of Harer \cite{Harer1986} which shows that the single non-trivial reduced homology group of the curve complex is the dualizing module for the mapping class group, thereby linking its group homology and cohomology.
One should not be too surprised at these dual functions for the curve complex, as the curve complex was originally proposed by Harvey \cite{Harvey} as an analog for the mapping class group of the spherical building at infinity for nonuniform lattices in semi-simple Lie groups.

Recent work of Aramayona and Leininger \cite{AL} advanced  the Ivanov side of the picture by giving a finite rigid subcomplex $\mathfrak{X}(\surface)$ in the curve complex with the property that any simplicial injection of $\mathfrak{X}(\surface)$ into the curve complex is induced by a unique mapping class.  
In this paper we relate the work in \cite{AL} to the Harer side  of that picture by relating the finite rigid set $\mathfrak{X}(\surface)$ of  Aramayona and Leininger to a homologically non-trivial sphere $X(\surface)$ in the curve complex whose orbit under the action of the mapping class group generates the reduced homology of the curve complex.  
First, we focus on the sphere with $n$ marked points.  In Proposition~\ref{prop:simpsphere} we find, for each $n\geq 5$, an explicit set of curves on $\surface^n_0$ whose associated vertices  in $\CC(\surface^n_0)$ determine an essential simplicial sphere in $\CC(\surface^n_0)$.  The information in Proposition~\ref{prop:simpsphere} will allow us to prove our first main result in $\S$\ref{sec:al}: 
\begin{theorem}
\label{thm:same}
Assume that $n\geq 5$.  Then the finite rigid set $\mathfrak{X}(\surface^n_0) \subset \CC (\surface_0^n)$ given in \cite{AL} is precisely the essential $(n-4)$-sphere 
$X(\surface^n_0) \subset \CC (\surface_0^n)$ whose reduced homology class is determined in Proposition~\ref{prop:simpsphere} of this paper to be a $\MCG(\surface_0^n)$-module generator of the reduced homology of the curve complex.    This answers Question 2 of \cite{AL} in the affirmative for $n\geq 5$.
\end{theorem}
Given the surprising coincidence of our essential sphere and the finite rigid set of Aramayona and Leininger for the surface $\surface_0^n$, and the fact that $X(\surface^6_0) = X(\surface^0_2) = \mathfrak{X}(\surface^0_2)$,  one might be tempted to conjecture that homologically non-trivial spheres in the curve complex are always finite rigid sets.  
Attempting to test our conjecture in the cases $\surface^0_g, \ g\geq 3$ we learned that it is not true.  In Proposition~\ref{prop:geng} we construct essential spheres $X(\surface^0_g)$ and $X(\surface^1_g)$ in $\CC(\surface^0_g)$ and $\CC(\surface^1_g)$ respectively.  Our second main result, proved in $\S$\ref{sec:gcomp} is:  
\begin{theorem}
\label{thm:notiff}
Let $g \geq 3$ and $n \in \{0,1\}$ or $g=2$ and $n=1$. Then the essential $(2g-2)$-sphere $X(\surface^n_g)\subset\CC(\surface_g^n)$ that is determined in Proposition~\ref{prop:geng}   
(i) represents a $\MCG(\surface_0^n)$-module generator for the reduced homology of $\CC(\surface^n_g)$, (ii) is a proper subset of $\mathfrak{X}(\surface^n_g)$,  but (iii) is not rigid.  
\end{theorem}
Modified versions of Theorems~\ref{thm:same} and \ref{thm:notiff} hold for the sporadic surfaces $\surface_0^4$, $\surface_1^0$ and $\surface_1^1$.  See $\S$\ref{SS:sporadic} for the precise statements.

While Theorem~\ref{thm:notiff} rules out the possibility that homologically non-trivial spheres in the curve complex must be rigid,  we have the following suggestive corollary:
\begin{corollary}
Suppose $g=0$ and $n \geq 5$ or $g \geq 2$ and $n \in \{0,1\}$. Let $ \mathfrak{X}( \surface_g^n) \subset \CC( \surface_g^n)$ be the finite rigid set given in \cite{AL}. Then the inclusion map $i: \mathfrak{X}( \surface_g^n) \to \CC( \surface_g^n)$ induces a homomorphism $i_\ast: \widetilde{\HH}_\ast (\mathfrak{X}( \surface_g^n)) \to \widetilde{\HH}_\ast (\CC( \surface_g^n))$ with non-trivial image.
\end{corollary}
\begin{proof}
In both cases the inclusion map $j:X \to \CC$ factors through the inclusion map  $i:\mathfrak{X} \to \CC$ and induces a homomorphism $j_\ast: \widetilde{\HH}_\ast (X) \to \widetilde{\HH}_\ast (\CC)$ with non-trivial image.
\end{proof}
\begin{question} Suppose $3g+n \geq 5$. Does the inclusion map  $i:  \mathfrak{X}( \surface_g^n) \to \CC( \surface_g^n)$ of the finite rigid set given in \cite{AL} induce a homomorphism $i_\ast: \widetilde{\HH}_\ast (\mathfrak{X}( \surface_g^n)) \to \widetilde{\HH}_\ast (\CC( \surface_g^n))$ with non-trivial image?
\end{question}

Here is a guide to the paper.  We begin, in $\S$\ref{sec:review of the background}, with a review of the necessary background.    The associahedron will enter into this review because it concisely describes the simplicial structure of the spheres in the curve complex which we demonstrate to be homologically non-trivial. We introduce it in $\S$\ref{sec:back}.

Our new work begins in $\S$\ref{sec:surfaces of genus 0}, where we restrict our discussion to surfaces of genus 0 with $n\geq 4$ marked points and give a finite $\MCG$-module resolution of the Steinberg module (the reduced homology of the curve complex) using the the cellular chain complex of the arc complex relative to the arc complex at infinity.
The last two terms of this resolution give a  $\MCG$-module presentation of this homology group which we use to give a single class whose orbit under the mapping class group generates the entire relative homology group. 
In $\S$\ref{sec:ccsph} we use Harer's homotopy equivalence discussed in $\S$\ref{sec:simpmap} to convert our non-trivial class in the homology of the arc complex relative to the arc complex at infinity to a class in the homology of the curve complex. We then simplify our representative of this class. 
In Proposition~\ref{prop:simpsphere} we give, explicitly, a finite subset $X$ of the curve complex $\CC$ the orbit of whose homology class under the action of the mapping class group generates the reduced homology of the curve complex. In $\S$\ref{sec:al} we find that the simplified representative is precisely the finite rigid set of Aramayona and Leininger, proving Theorem~\ref{thm:same}.

In $\S$\ref{sec:genusg} we consider the same situation when $g\geq 1$ and $n\in \{0, 1\}$.  This situation was already treated in 
\cite{Broaddus}, however we are now able to simplify the results given there drastically.  We give, in Proposition~\ref{prop:geng},  an explicit  simplified homologically non-trivial  sphere in the curve complex the orbit of whose homology class under the action of the mapping class group generates the reduced homology of the curve complex. 
We  establish that it is a proper subset of the finite rigid set of Aramayona and Leininger.   In  $\S$\ref{sec:gcomp} we prove  Theorem~\ref{thm:notiff}, by showing that it is not a finite rigid set.

Modified versions of Theorems~\ref{thm:same} and \ref{thm:notiff}, for sporadic surfaces, are discussed and proved in $\S$\ref{SS:sporadic}.

{\bf Acknowledgement:}  \  We gratefully acknowledge the help of the anonymous referee, who  suggested significant simplifications to the work in $\S$\ref{sec:ggen} and $\S$\ref{sec:gcomp}.


\section{Review of the Background}
\label{sec:review of the background}

Our results in $\S$\ref{sec:surfaces of genus 0} and $\S$\ref{sec:genusg}  will concern the surfaces $\surface_0^n$ where $n \geq 4$ and $\surface_g^n$ where $g \geq 1$ and $n \in \{ 0, 1\}$.  However, many of the results that we review in this section apply in greater generality so we provide the stronger statements and more general definitions when possible.

\subsection{The curve complex,  arc complex and associahedron}
\label{sec:back}

Let $n \geq 4$ or $g \geq 1$ and set
\begin{equation}
\label{eq:tau}
\tau =  \left \{\begin{array}{ll} n-4, & g = 0 \\ 2g-2, & g \geq 1 \mbox{ and } n = 0 \\ 2g+n-3, & g \geq 1 \mbox{ and } n \geq 1\end{array} \right. .
\end{equation}
By the work of Harer we have the following theorem:
\begin{fact}[{\cite[Theorem 3.5]{Harer1986}}]
\label{fact:cc}
Assume $n \geq 4$ or $g \geq 1$.
The curve complex $\CC(\surface_g^n)$ has the homotopy type of a countably infinite wedge sum of spheres of dimension $\tau$ thus
\[
\CC(\surface_g^n) \simeq \vee^\infty S^{\tau} .
\]
where $\tau$ is as in (\ref{eq:tau}).
\end{fact}
In particular the reduced homology $\widetilde{\HH}_\ast(\CC(\surface_g^n);\Z)$ of the curve complex is non-trivial only in dimension $\tau$ where we get
\[
\widetilde{\HH}_{\tau}(\CC(\surface_g^n);\Z) \cong \oplus^\infty \Z .
\]

The {\em dualizing module} of a duality group $\Gamma$ of cohomological dimension $d$ is a $\Gamma$-module $D$ such that for any $\Gamma$-module $A$ and any $k \in \Z$ we have
\[
\HH^k(\Gamma;A) \cong \HH_{d-k}(\Gamma; A \otimes_{\Z} D)
\]
where $A \otimes_{\Z} D$ has the diagonal module structure. (See \cite[VIII.10]{Brown} for a general introduction to duality groups.)

In particular if $\Gamma$ is any torsion free finite index subgroup of the mapping class group $\MCG(\surface_g^n)$ where $n \geq 4$ or $g \geq 1$ then by the work of Harer \cite[Theorem 4.1]{Harer1986} $\Gamma$ is a duality group with dualizing module the {\em Steinberg module}
\[
\St = \St(\surface_g^n) := \widetilde{\HH}_{\tau}(\CC(\surface_g^n);\Z) .
\]
We view this as motivation for a careful study of the homotopy type of the curve complex.

To help us in this study, it will be convenient to introduce a closely related complex,  the ``arc complex" of a surface with marked points.
Let $\surface_g^n$ be the surface with genus $g \geq 0$ and $n \geq 1$ marked points and let $V \subset \surface_g^n$ be the set of marked points.  
An {\em arc} in $\surface_g^n$ is either the unoriented image of a properly embedded path in $\surface_g^n$ joining two points of $V$ which is disjoint from $V$ except at its endpoints or else the image of an unoriented simple loop in $\surface_g^n$ based at a point in $V$ disjoint from $V$ except at its basepoint.  
An {\em essential arc} in  $\surface_g^n$ is an arc which does not bound an embedded disk in $\surface_g^n$ whose interior is disjoint from $V$.  We will say that two arcs are {\em disjoint} if they are disjoint except possibly at their endpoints.  
The isotopy class of an essential arc is its orbit under the elements of $\Homeo^0(\surface_g^n)$ which fix the marked points pointwise.  
An  {\em arc system} is a set of isotopy classes of arcs with pairwise disjoint representatives and a {\em sub arc system} of an arc system is any non-empty subset of the arc system.  
The {\em arc complex} $\AC = \AC (\surface_g^n)$ is the simplicial complex whose vertices are isotopy classes of essential arcs and whose $k$-simplices correspond to arc systems with $k+1$ isotopy classes of arcs.  
Observe that a maximal arc system gives a triangulation of the surface $\surface_g^n$ with $n$ vertices so by Euler characteristic it will have $6g+3n-6$ arcs.  Therefore we have
\[
\dim (\AC (\surface_g^n)) = 6g+3n-7.
\]

The arc complex has a very nice property:
\begin{fact}[\cite{Hatcher}]
\label{fact:con}
If $n \geq 2$ or if both $g \geq 1$ and $n \geq 1$ then $\AC (\surface_g^n)$ is contractible.
\end{fact}

An arc system {\em fills} the surface $\surface_g^n$ if the arcs in the arc system cut the surface into disks with at most one marked point in their interior.  
The {\em arc complex at infinity} $\AC_\infty = \AC_\infty (\surface_g^n)$ is the union of all simplices of $\AC (\surface_g^n)$ whose corresponding arc systems do {\em not} fill the surface.  
One may also alternatively describe $\AC_\infty (\surface_g^n)$ as the subset of points in $\AC (\surface_g^n)$ with infinite stabilizers under the action of $\MCG (\surface_g^n)$.

Now suppose that $g=0$ and $n \geq 4$ or $g \geq 1$ and $n = 1$ and let $\tau$ be as in equation (\ref{eq:tau}).  At least $\tau + 2$ pairwise disjoint arcs are needed to fill the surface $\surface_g^n$. Hence $\AC_\infty (\surface_g^n)$ contains the entire $\tau$-skeleton of  $\AC (\surface_g^n)$.  
If on the other hand we have $g \geq 1$ and $n \geq 2$ then at least $\tau + 1$ pairwise disjoint arcs are needed to fill the surface which implies that $\AC_\infty (\surface_g^n)$ contains the $(\tau -1)$-skeleton of  $\AC (\surface_g^n)$ but not the entire  $\tau$-skeleton of  $\AC (\surface_g^n)$.  
This fact renders an essential part of our argument given in $\S$\ref{sec:res} applicable only to the cases where $g=0$ and $n \geq 4$ or $g \geq 1$ and $n = 1$.

We now introduce a simplicial complex which we will later embed in the curve complex as a homologically non-trivial sphere.  Let $m \geq 3$.  The associahedron $K_m$ is a convex polytope homeomorphic to the closed $(m-2)$-ball. 
Its boundary $\partial K_m$ is a cell complex homeomorphic to the $(m-3)$-sphere.  The dual of the cell complex $\partial K_m$ is a {\em simplicial} complex $D_m$ which is again homeomorphic to the $(m-3)$-sphere.  Here we provide a convenient description of $D_m$ taken from \cite[Corollary 2.7]{CD}.

Let $\Theta$ be a graph.  A {\em tube} is any proper non-empty subset of the vertices of $\Theta$ whose induced subgraph in $\Theta$ is connected. Let $t_1$ and $t_2$ be two tubes of $\Theta$.  We say that $t_1$ and $t_2$ are {\em nested} if $t_1 \subset t_2$ or  $t_2 \subset t_1$.  
The tubes $t_1$ and $t_2$ {\em overlap} if $t_1 \cap t_2 \neq \varnothing$ and $t_1$ and $t_2$ are not nested.  Tubes $t_1$ and $t_2$ are {\em adjacent} if $t_1 \cap t_2 = \varnothing$ and  $t_1 \cup t_2$ is a tube.
Two tubes in $\Theta$ are {\em compatible} if they do not overlap and are not adjacent.
\begin{definition}
Let $\mathscr{D}(\Theta)$ be the simplicial complex with a vertex for each tube of $\Theta$ and a $k$-simplex for each set of $k+1$ pairwise compatible tubes in $\Theta$.
\end{definition}
\begin{theorem}[{\cite[Corollary 2.7]{CD}}]
\label{thm:assoc}
The dual $D_m$ of the boundary $(m-3)$-sphere of the associahedron $K_m$ is the simplicial complex $\mathscr{D}(\Lambda_{m-1})$ where $\Lambda_{m-1}$ is the path graph with $m-1$ vertices given in Figure~\ref{fig:pathonly}.
\end{theorem}
\begin{figure}[htbp]
\begin{center}
\includegraphics{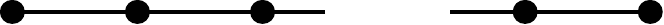}
\put(-191,11){$1$}
\put(-154.5,11){$2$}
\put(-118,11){$3$}
\put(-90,-1){\Large$\cdots$}
\put(-53,11){$m-2$}
\put(-17,11){$m-1$}
\end{center}
\caption{A path graph $\Lambda_{m-1}$ with vertex set $\{ 1, 2, \cdots, m-1 \}$.\label{fig:pathonly}}
\end{figure}
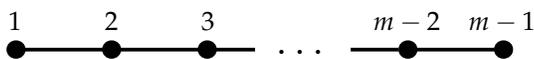
\begin{example}
$D_5$ is the simplicial complex homeomorphic to $S^2$ depicted in Figure~\ref{fig:assoc}.
\end{example}
\begin{figure}[htbp]
\begin{center}
\includegraphics{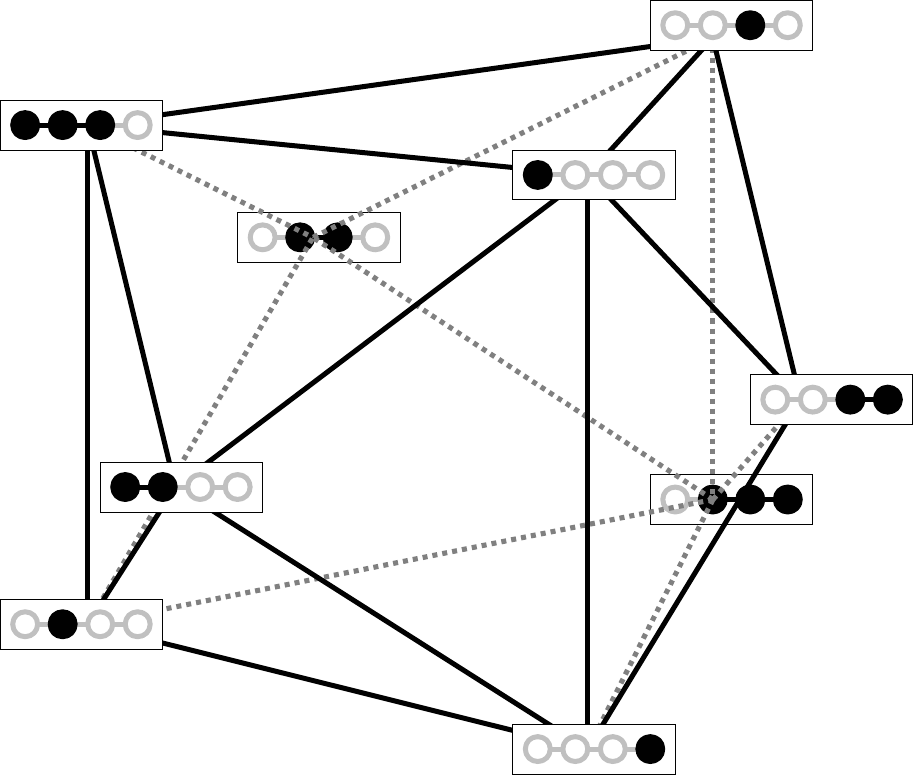}
\end{center}
\caption{The simplicial complex $D_5$ is the dual of the boundary of the associahedron $K_5$\label{fig:assoc}}
\end{figure}
The curious reader will discover that the literature on the combinatorics of simplicial complexes is filled with references to the associahedron.  
In particular, the reader may wish to compare Figure~\ref{fig:assoc} with  Figure 10 of \cite{Broaddus}, which illustrates the role that the associahedron will play in this paper.  
Indeed, it is used to describe both the finite rigid subsets sets of $\CC(\surface^6_0)$ and of $\CC(\surface^0_2)$, illustrating the well-known isomorphism between these two complexes \cite{Luo}.

\subsection{Maps between simplicial complexes}
\label{sec:simpmap}

In the discussion that follows we will encounter a number of maps between simplicial complexes which are not simplicial maps but are nonetheless very well-behaved piecewise linear maps.   
For example we will describe below a continuous map from the arc complex at infinity $\AC_\infty ( \surface_g^n )$ to the curve complex $\CC ( \surface_g^n )$ which induces a homotopy equivalence (see Known Result~\ref{fact:accc}). 
Here we develop some convenient language and observations for working with these maps.

If $X$ is a simplicial complex with vertex set $X^0$ we will specify a general point $p \in X$ as a linear combination
\[
p = \sum_{v \in X^0} p_v v
\]
where $\sum_{v \in X^0} p_v =1$ and for all $v \in X^0$ we have $0 \leq p_v \leq 1$.  For a point $p \in X$ the {\em support} of $p$ is the set
\[
\supp p = \{ v \in X^0 | p_v > 0 \}
\]
and for a subset $U \subset X$ we set $\supp U = \bigcup_{p \in U} \supp p$.
Our simplicial complexes will all be finite dimensional so we will have $\dim X = \max_{p \in X} |\supp p | - 1$.  Conversely, given a set $S \subset X^0$ we will define the {\em span} of $S$ to be the set
\[
\Span S = \{ p \in X | \supp p \subset S \}.
\]

\begin{definition}[Interpolability]
\label{def:inter}
Let $X$ and $Y$ be simplicial complexes with vertex sets $X^0$ and $Y^0$ respectively.  Let $f^0: X^0 \to Y$ be a function.  We say that $f^0$ is {\em interpolable} if for any simplex $\sigma$ of $X$ there is some simplex $\eta$ of $Y$ such that for every vertex $v$ of $\sigma$ we have $f^0(v) \in \eta$.
\end{definition}
Note that it suffices to verify the condition of Definition~\ref{def:inter} for the maximal simplices of a finite dimensional simplicial complex $X$.

\begin{definition}[Linear interpolation]
\label{def:lin}
Let $X$ and $Y$ be simplicial complexes with vertex sets $X^0$ and $Y^0$ respectively.  We say that $f: X \to Y$ is a {\em linear interpolation} if the restriction $f|_{X^0} : X^0 \to Y$ is interpolable and for any point $p =  \sum_{v \in X^0}  p_v v \in X$ we have
\[
f\left( \sum_{v \in X^0}  p_v v \right) =\sum_{v \in X^0}  p_v f(v).
\]
\end{definition}

\begin{remark}
Note that any interpolable function $f^0: X^0 \to Y$ from the $0$-skeleton $X^0$ of a simplicial complex $X$ to a simplicial complex $Y$ extends to a unique linear interpolation $f:X \to Y$ via the equation
\[
f\left( \sum_{v \in X^0}  p_v v \right) =\sum_{v \in X^0}  p_v f^0(v).
\]
\end{remark}

\begin{lemma}[Linear interpolations are continuous]
\label{lem:lincont}
Let $f: X \to Y$ be a linear interpolation.  Then $f$ is continuous.
\end{lemma}

\begin{proof}
This follows from the gluing lemma from basic topology and the observation that linear interpolations agree on intersections of simplices.
\end{proof}

\begin{lemma}
\label{lem:quil}
Let $f,g:X \to Y$ be linear interpolations and suppose that for each simplex $\sigma$ of $X$ there is a simplex $\eta$ of $Y$ such that for each vertex $v$ of $\sigma$ we have $f(v),g(v) \in \eta$. Then $f$ and $g$ are homotopic functions.
\end{lemma}

\begin{proof}
For $t \in [0,1]$ define $F_t^0: X^0 \to Y$ by setting
\[
F_t^0(v) = (1-t)f(v) + tg(v)
\]
for each vertex $v \in X^0$.  Note that $F_t^0$ is interpolable so we may extend it to a linear interpolation $F_t:X \to Y$.  Now define $F:X \times I \to Y$ by setting $F(x,t) = F_t(x)$.  One observes that $F(x,0) = f(x)$ and $F(x,1) = g(x)$ for all $x \in X$.  The continuity of $F$ follows from the continuity of the restriction $F|_{\sigma \times I}$ for each simplex $\sigma$ of $X$ and the gluing lemma. 
\end{proof}

\begin{corollary}
\label{cor:drop}
Let $f: X \to Y$ be a linear interpolation.   For each vertex $v \in X^0$ choose a vertex $w_v \in \supp f(v)$  and define $h^0:X^0 \to Y$ by setting $h^0(v) = w_v$.  Then $h^0$ is interpolable and $f$ is homotopic to the linear interpolation $h: X \to Y$ of $h^0$.
\end{corollary}
\begin{proof}
Let $\sigma$ be a simplex of $X$.  Then since $f$ is a linear interpolation there is a simplex $\eta$ of $Y$ such that $f(\sigma) \subset \eta$ and hence $\supp f(\sigma) \subset \eta$.  
By construction if $v$ is a vertex of $\sigma$ then $\supp h^0(v) \subset \supp f(\sigma)$ so $\supp h^0(v) \subset \eta$.  It follows that $h^0$ is interpolable with linear interpolation $h:X \to Y$.  Moreover, $h(v),f(v) \in \eta$ so we may apply Lemma~\ref{lem:quil} to conclude that $h$ and $f$ are homotopic.
\end{proof}

Finally we note that if $X$ is a simplicial complex then its {\em barycentric subdivision} $\bs X$ is the simplicial complex with $0$-skeleton given by the set
\[
(\bs X)^0 = \{ \sigma | \mbox{ $\sigma$ is a simplex of $X$} \}
\]
and with $n$-simplices given by flags of simplices of $X$ of length $n+1$.  One gets a homeomorphism between $X$ and $\bs X$ by letting
$j^0: (\bs X)^0 \to X$
be the function sending the vertex $\sigma \in (\bs X)^0$ to the point $\frac{1}{|\supp \sigma|}\sum_{v \in \supp \sigma} v$ and letting 
\[
j: \bs X \to X
\]
be the linear interpolation of $j^0$.  The {\em barycentric subdivision map}
\begin{equation}\label{eq:bs}
\bs :X \to \bs X
\end{equation}
is the inverse of $j$.

We end our general discussion of linear interpolations with the following observation:
\begin{remark}
\label{rem:simp}
A linear interpolation $f: X \to Y$ is a simplicial map if and only if for all $v \in X^0$ we have $| \supp f(v) | = 1$.
\end{remark}

With the language of linear interpolations we now review Harer's homotopy equivalence between the arc complex at infinity and the curve complex.   Suppose that $g \geq 1$ and $n \geq 1$ or $g=0$ and $n \geq 4$.  Let $\AC_\infty = \AC_{\infty}( \surface_g^n)$ and $\CC = \CC( \surface_g^n)$.  Let $\bs\AC_\infty$ be the barycentric subdivision of $\AC_\infty$.  The vertices of $\bs\AC_\infty$ will correspond to non-filling arc systems with $k$-simplices corresponding to increasing flags of non-filling arc systems of length $k+1$.

When an arc system $\alpha$ does not fill the surface $\surface_g^n$ we can associate a curve system $\gamma(\alpha)$ to it by letting $N(\alpha) \subset \surface_0^n$ be a closed regular neighborhood of the union of the arcs in the arc system $\alpha$ and letting the curve system $\gamma(\alpha)$ be the set of essential (in $\surface_g^n$) boundary curves of the subsurface $\surface_g^n  -  N(\alpha)$ with duplicate curves removed.  Let
\[
\Phi^0 : (\bs \AC_\infty )^0 \to \CC
\]
be the function sending the non-filling arc system $\alpha$ to the barycenter $\frac{1}{|\gamma(\alpha)|}\sum_{c \in \gamma(\alpha)} c$ of the simplex in $\CC$ corresponding to the the curve system $\gamma(\alpha)$.

One sees that $\Phi^0$ is interpolable as follows.  If $\sigma$ is a simplex of $\bs \AC_\infty$ then it corresponds to an increasing flag of non-filling arc systems $\alpha_1 \subsetneq \alpha_2 \subsetneq \cdots \subsetneq \alpha_k$.  
We may arrange that the corresponding regular neighborhoods satisfy $N(\alpha_i) \subset \inte N(\alpha_{i+1})$ when $1 \leq i < k$.  
Thus all boundary curves can me made to be simultaneously disjoint.  Let $\eta$ be the simplex of $\CC$ corresponding to the curve system which is the union of all such sets of boundary curves.  If $v$ is a vertex of $\sigma$ then $\Phi^0(v) \in \eta$.  Thus $\Phi^0$ is interpolable.  Let
\begin{equation}
\label{eqn:phi}
\Phi : \bs \AC_\infty \to \CC
\end{equation}
be the linear interpolation of $\Phi^0$.

\begin{fact}[{\cite[Theorem 3.4]{Harer1986}}]
\label{fact:accc}
Suppose that $g \geq 1$ and $n \geq 1$ or $g=0$ and $n \geq 4$.  Then the map $\Phi$ in (\ref{eqn:phi}) induces a homotopy equivalence between $ \AC_\infty ( \surface_g^n )$ and  $\CC ( \surface_g^n )$.
\end{fact}

\subsection{A free resolution of the homology of the Steinberg module}
\label{sec:res}

We now restrict to the cases where either $g=0$ and $n\geq 4$ or $g\geq 1$ and $n=1$ and utilize techniques similar to those in \cite{Broaddus} to give a $\MCG(\surface_g^n)$-module resolution of $\St (\surface_g^n)$. It is the failure of (\ref{eq:tail}) below when $g \geq 1$ and $n \geq 2$ which renders our argument invalid for that case.

By Known Result~\ref{fact:con} the arc complex $\AC = \AC(\surface_g^n)$ is contractible and by Known Result~\ref{fact:accc} the arc complex at infinity $\AC_\infty = \AC_\infty(\surface_g^n)$ is homotopy equivalent to the curve complex $\CC = \CC(\surface_g^n)$.  
Known Result~\ref{fact:cc} says that curve complex has the homotopy type of a wedge of $\tau$-dimensional spheres.  Therefore for $k \geq 1$
\begin{equation}
\label{eq:sphhom}
\begin{split}
\HH_{k}(\AC,\AC_\infty;\Z) & \cong \widetilde{\HH}_{k-1}( \AC_\infty;\Z) \\
 & \cong \widetilde{\HH}_{k-1}( \CC;\Z) \\
 & = \left \{\begin{array}{ll} 0, & k \neq \tau + 1 \\ \St(\surface_g^n), & k=\tau+1.\end{array} \right.
\end{split}
\end{equation}
Now consider the cellular chain complex $\CH_\ast(\AC,\AC_\infty)$ for the pair of spaces $(\AC, \AC_\infty)$.
\[
 \cdots \to 0 \to \CH_{6g+3n-7}(\AC, \AC_\infty) \to  \cdots  \to \CH_{1}(\AC, \AC_\infty) \to \CH_{0}(\AC , \AC_\infty) \to 0
\]
where as usual we define $\CH_{k}(\AC, \AC_\infty) := \CH_{k}(\AC) / \CH_{k}(\AC_\infty)$.  A chain complex is an exact sequence when all of its homology groups are $0$.  
Note that every homology group of the pair of spaces $(\AC, \AC_\infty)$ is zero except for $\HH_{\tau + 1}(\AC, \AC_\infty;\Z)$ so the chain complex $\CH_{\ast}(\AC, \AC_\infty)$ is very close to being an exact sequence.

No arc system with $2g-1$ or fewer arcs can fill the surface $\surface_g^1$ and no arc system with $n-3$ or fewer arcs can fill the surface $\surface_0^n$.  
Hence in both cases no arc system with $\tau + 1$ or fewer arcs can fill the surface so the entire $\tau$-skeleton of $\AC$ is contained in $\AC_\infty$.  Thus
\begin{equation}
\label{eq:tail}
\CH_{\tau}(\AC, \AC_\infty) := \CH_{\tau}(\AC) / \CH_{\tau}(\AC_\infty) =  \CH_{\tau}(\AC_\infty) / \CH_{\tau}(\AC_\infty) = 0.
\end{equation}

From equations (\ref{eq:sphhom}) and (\ref{eq:tail})  we have that
\begin{equation}
\label{eq:stdes}
\begin{split}
\St(\surface_g^n) &\cong \HH_{\tau + 1}(\AC , \AC_\infty) \\
& = \frac{\ker \bigg( \CH_{\tau + 1}(\AC, \AC_\infty) \to \CH_{\tau}(\AC , \AC_\infty) \bigg)}{\im \bigg(\CH_{\tau + 2}(\AC, \AC_\infty) \to \CH_{\tau + 1}(\AC , \AC_\infty) \bigg)}\\
& = \frac{\ker \bigg( \CH_{\tau + 1}(\AC, \AC_\infty) \to 0 \bigg)}{\partial \CH_{\tau + 2}(\AC, \AC_\infty)}\\
& = \frac{\CH_{\tau + 1}(\AC, \AC_\infty)}{\partial \CH_{\tau + 2}(\AC, \AC_\infty)}\\
& \cong \frac{\CH_{\tau + 1}(\AC)}{\CH_{\tau + 1}(\AC_\infty) \oplus \partial \CH_{\tau + 2}(\AC)}\\
\end{split}
\end{equation}
where $\partial$ is the boundary operator for  $\CH_\ast(\AC,\AC_\infty)$ or $\CH_\ast(\AC)$ depending on the context.  
We thus have the following very useful description of the homology of the curve complex:
\begin{proposition}
\label{prop:stpres}
Let $g = 0$ and $n\geq 4$ or $g \geq 1$ and $n=1$.  Let $\tau$ be as in (\ref{eq:tau}) and  $\CH_\ast(\AC)$ and $\CH_\ast(\AC_\infty)$ be the simplicial chain complexes of the arc complex $\AC = \AC (\surface_g^n)$ and arc complex at infinity $\AC_\infty= \AC_\infty (\surface_g^n)$ respectively.  Then the Steinberg module $\St(\surface_g^n) := \widetilde{\HH}_{\tau}(\CC(\surface_g^n);\Z)$ has the following presentation:
\[
\St(\surface_g^n)  \cong \frac{\CH_{\tau + 1}(\AC)}{\CH_{\tau + 1}(\AC_\infty) \oplus \partial \CH_{\tau + 2}(\AC)}.
\]
\end{proposition}
Proposition~\ref{prop:stpres} says that the reduced homology of the curve complex is the quotient of the free abelian group on all arc systems in $\surface_g^n$ with $\tau + 2$ arcs modulo all {\em non-filling} arc systems with $\tau + 2$ arcs and boundaries of arc systems with $\tau + 3$ arcs.

Notice that if $k \geq \tau + 2$ then by equation (\ref{eq:sphhom}) we have $\HH_{k}(\AC,\AC_\infty) = 0$ so
\begin{equation}
\label{eq:res}
 \cdots \to 0 \to \CH_{6g+3n-7}(\AC, \AC_\infty) \to \cdots \to \CH_{\tau + 2}(\AC, \AC_\infty) \to \CH_{\tau + 1}(\AC , \AC_\infty) 
\end{equation}
is an {\em exact} sequence.  As with any terminating exact sequence we can append two more terms to get a slightly longer exact sequence
\[
\cdots \to 0 \to \CH_{6g+3n-7}(\AC, \AC_\infty) \to \cdots  \to \CH_{\tau + 2}(\AC, \AC_\infty) \to \CH_{\tau + 1}(\AC , \AC_\infty) \to \frac{\CH_{\tau + 1}(\AC, \AC_\infty)}{\partial \CH_{\tau + 2}(\AC, \AC_\infty)} \to 0.
\]
From (\ref{eq:stdes}) we know that the penultimate term of this sequence is isomorphic to $\St(\surface_g^n)$.
Since the sequence in (\ref{eq:res}) is exact we have a $\MCG(\surface_g^n)$-module resolution of $\St(\surface_g^n)$.  
Note that a filling arc system in $\surface_g^n$ always has a finite but possibly non-trivial stabilizer in $\MCG(\surface_g^n)$.  
Hence for $k \geq \tau + 1$ the relative chain group $\CH_{k}(\AC, \AC_\infty)$ is not quite a free $\MCG(\surface_g^n)$-module.  However, if one restricts to a torsion free subgroup $\Gamma < \MCG(\surface_g^n)$ then $\CH_{k}(\AC, \AC_\infty)$ will be a free $\Gamma$-module.

\begin{proposition}
\label{prop:res}
Let $g = 0$ and $n\geq 4$ or $g \geq 1$ and $n=1$.  Let $\tau$ be as in (\ref{eq:tau}).  The Steinberg module $\St(\surface_g^n) := \widetilde{\HH}_{\tau}(\CC(\surface_g^n);\Z)$ has the following $\MCG(\surface_g^n)$-module resolution of finite length:
\[
  \cdots \to 0 \to \CH_{6g+3n-7}(\AC, \AC_\infty) \to \cdots \to \CH_{\tau + 2}(\AC, \AC_\infty) \to \CH_{\tau + 1}(\AC , \AC_\infty)  \to \St(\surface_g^n) .
\]
This resolution is not free but will be free for any torsion free subgroup $\Gamma < \MCG(\surface_g^n)$.  The terms of this resolution will be finitely generated as $\Gamma$-modules if the index of $\Gamma$ in $\MCG(\surface_g^n)$ is finite.
\end{proposition}

\section{Essential spheres and finite rigid sets for the surfaces $\surface^n_0$}
\label{sec:surfaces of genus 0}

We begin the new work in this paper.  In this section we assume that $n \geq 4$, and let $\surface_0^n$ be the surface of genus $0$ with marked point set $V \subset \surface_0^n$ of size $n$.  We turn our attention to the goal of finding a single class in $\St(\surface_0^n)$  whose orbit under $\MCG (\surface_0^n)$ generates $\St(\surface_0^n)$ as a $\Z$-module.

From Proposition~\ref{prop:stpres} we have that
\[
\St(\surface_0^n)  \cong \frac{\CH_{n-3}(\AC)}{\CH_{n-3}(\AC_\infty) \oplus \partial \CH_{n-2}(\AC)}.
\]
$\CH_{n-3}(\AC)$ is the free abelian group on all arc systems with $n-2$ arcs.  
The mapping class group is not transitive on the set of these arc systems, but there are a finite number of orbits.  
We will show that after quotienting by the relations $\CH_{n-3}(\AC_\infty)$ and $\partial \CH_{n-2}(\AC)$ that every element of $\St(\surface_0^n)$ will be a linear combination of elements of the $\MCG (\surface_0^n)$-orbit of the class of a single arc system.

The set of non-filling arc systems with $n-2$ arcs forms a basis for  $\CH_{n-3}(\AC_\infty)$.  This is a subset of the set of all arc systems with $n-2$ arcs which forms a basis for  $\CH_{n-3}(\AC)$.  Thus we see that non-filling arc systems give the trivial class in  $\St(\surface_0^n)$.

The union of the arcs in an arc system $\alpha$ is a graph $G(\alpha)$.  Note that in general the vertices of $G(\alpha)$ will be only those marked points in $V$ which are endpoints of at least one arc in the arc system $\alpha$. 
The arc system $\alpha$ fills the surface if $\surface_0^n  -  G(\alpha)$ is a disjoint union of open disks containing at most one marked point.  
If the filling arc system has exactly $n-2$ arcs then each of these disks must contain exactly one marked point.  We note that if $\alpha$ is a filling arc system then $G(\alpha)$ must be connected since its complement is a disjoint union of open disks.  	
\begin{lemma}
\label{lem:path}
If $\alpha$ is a filling arc system in $\surface_0^n$ with $n-2$ arcs then its class in the Steinberg module $\St(\surface_0^n)$ is a $\Z$-linear combination of the classes of filling arc systems $\beta$ with $n-2$ arcs such that the graph $G(\beta)$ is a path graph.
\end{lemma}

\begin{proof}
Let $\alpha$ be an arc system in $\surface_0^n$ with $n-2$ arcs.  If $\alpha$ does not fill then its class in $\St(\surface_0^n)$ is trivial.  
Define a {\em tine} of the graph $G(\alpha)$ to be a path subgraph $P$ of $G(\alpha)$ which is either any single vertex of $G(\alpha)$ or else has at least one vertex which has valence one in $G(\alpha)$ and exactly two vertices which have valence not equal to 2 in $G(\alpha)$.  (See sketch (ii) of Figure~\ref{fig:tine}).  
Note that every vertex of $G(\alpha)$ is a tine of length $0$ and that distinct tines of $G(\alpha)$ can intersect in at most a single vertex.  Define the complexity $\mu(\alpha)$ of the arc system $\alpha$ to be the length of a longest tine of $G(\alpha)$.    

Let $P(\alpha)$ be a tine of $\alpha$ with maximal length.  If $\mu(\alpha) = n-2$ then $G(\alpha)$ is a path graph and we are done.  Assume that $\mu(\alpha) < n-2$.   
If $\mu(\alpha) = 0$ let $v$ be the unique vertex of $P(\alpha)$.  Otherwise let $v$ be the unique vertex of $P(\alpha)$ which has valence 1 in $G(\alpha)$.  
The vertex $v$ must be in the closure of at least one component of $\surface_0^n - G(\alpha)$.  Let $w$ be the marked point contained in the interior of such a component. (See sketch (ii)  of Figure~\ref{fig:tine}).  
Let $\hat{a}$ be an arc disjoint from the arcs of $\alpha$ connecting $v$ to $w$. Let $\hat{\alpha} = \alpha \cup \{\hat{a} \}$ which is an arc system with $n-1$ arcs.  (See sketch (iii) of Figure~\ref{fig:tine}).
\begin{figure}[htbp] 
\begin{center}
\framebox{
\includegraphics{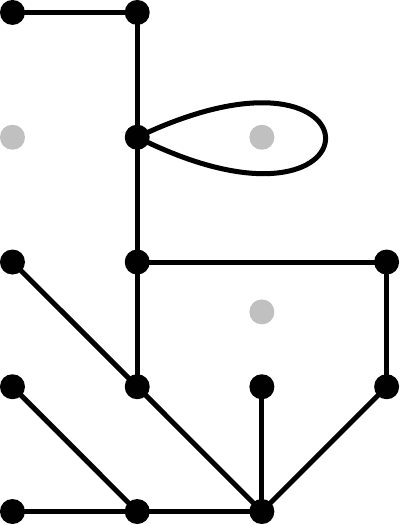}
\put(-75,-15){(i) \ $G(\alpha)$}
}
\hspace{0.2in}
\framebox{
\includegraphics{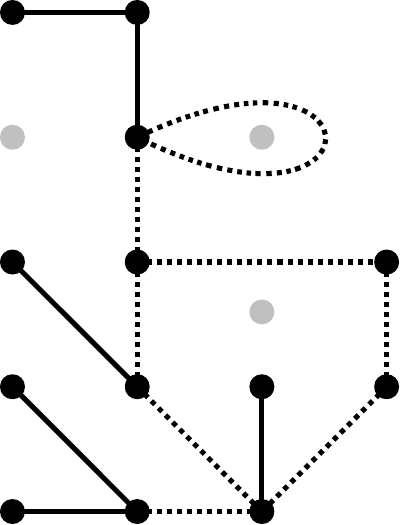}
\put(-95,-15){(ii) \ Tines of $G(\alpha)$}
\put(-106,110){$w$}
\put(-99,134){$P(\alpha)$}
\put(-114,138){$v$}
}
\hspace{0.2in}
\framebox{
\includegraphics{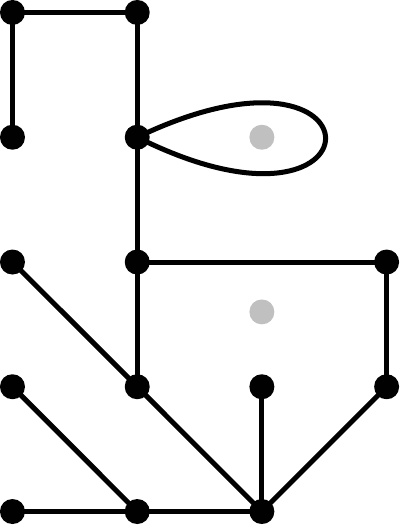}
\put(-78,-15){(iii) \ $G(\hat{\alpha})$}
\put(-106,110){$w$}
\put(-110,126){$\hat{a}$}
}
\caption{In sketch (i) $\alpha$ is a filling arc system in the surface $\surface_0^{n}$ for $n=16$ with $n-2=14$ arcs. The complexity is $\mu(\alpha)=2$.    
In sketch (ii) the solid part of the graph is the union of the tines of $G(\alpha)$.  
The graph $G(\alpha)$ has 13 tines of length zero, four tines of length one and one tine $P(\alpha)$ of length two.  
In sketch (iii) we have a filling arc system $\hat{\alpha}$ such that the relation in  $\St (\surface_0^{16})$ induced by $\partial \hat{\alpha}$ expresses the class of $\alpha$ as a linear combination of classes of filling arc systems all with with complexity at least $3$.}
\label{fig:tine}
\end{center}
\end{figure}
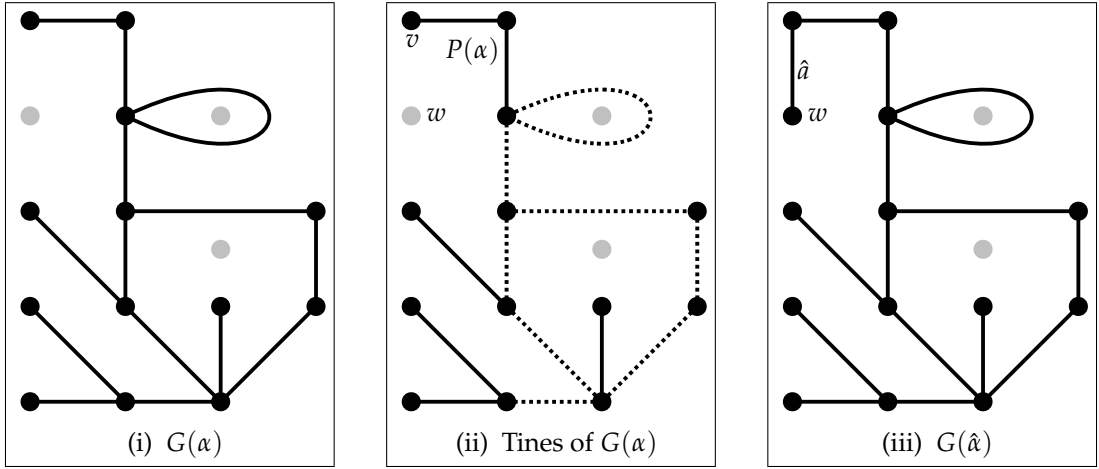

Given $b \in \hat{\alpha}$ set $\alpha_b = \hat{\alpha} - \{ b \}$.  We claim exactly one of the following holds:
\begin{enumerate}
\item $\alpha_b = \alpha$.
\item $\alpha_b$ is a non-filling arc system in $\surface_0^n$.
\item $\alpha_b$ is a filling arc system with $\mu(\alpha_b)  > \mu(\alpha)$.
\end{enumerate}
Note that $\alpha_b = \alpha$ if and only if $b = \hat{a}$.  Suppose $b \neq \hat{a}$.  Then $b \in \alpha$.  If $b \subset P(\alpha)$ then $G(\alpha_b)$ is disconnected so $\alpha_b$ is non-filling.  
If $b \not\subset P(\alpha)$ then $\alpha_b$ may or may not be a filling system but in either case $\mu(\alpha_b) \geq \mu(\alpha) +1$.

We may use the relation $\partial \hat{\alpha}$ in $\St(\surface_0^n)$ to express the class of the arc system $\alpha$ as a $\Z$-linear combination of classes of arc systems with higher complexity.  
By induction we may write any such class as a $\Z$-linear combination of classes of filling arcs systems with the maximal complexity $n-2$.  The graph for each of these arc systems must be a path graph.
\end{proof}

\begin{proposition}
\label{prop:gen0}
Suppose $n \geq 4$.  Let $\theta_\upsilon \subset \AC(\surface_0^n)$ be the $(n-3)$-simplex corresponding to the arc system $\upsilon$ in Figure~\ref{fig:path}.  Then the class
\[
[\theta_\upsilon] \in \St (\surface_0^n) = \frac{\CH_{n-3}(\AC)}{\CH_{n-3}(\AC_\infty) \oplus \partial \CH_{n-2}(\AC)}
\]
is non-trivial and generates $\St (\surface_0^n)$ as a $\MCG(\surface_0^n)$-module.
\end{proposition}

\begin{figure}[htbp]
\begin{center}
\includegraphics{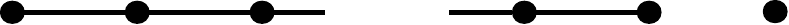}
\put(-209,9){$y_1$}
\put(-173,9){$y_2$}
\put(-69,9){$y_{n-2}$}
\put(-9,11){$p_{\upsilon}$}
\put(-126,-1){\Large$\cdots$}
\end{center}
\caption{A filling arc system $\upsilon = \{ y_1, y_2, \cdots, y_{n-2} \}$ with  $n-2$ arcs in the surface $\surface_0^n$ whose class generates $\St (\surface_0^n)$ as a $\MCG(\surface_0^n)$-module.\label{fig:path}}
\end{figure}

\begin{proof}
By Lemma~\ref{lem:path} the class of any arc system with $n-2$ arcs is a $\Z$-linear combination of classes of arc systems whose graphs are path graphs.  
All path graphs with $n-2$ edges in $\surface_0^n$ are in the  $\MCG(\surface_0^n)$-orbit of the arc system $\upsilon$ pictured in Figure~\ref{fig:path}.  
Hence the $\MCG(\surface_0^n)$-orbit of the class of $\theta_\upsilon$ generates $\St (\surface_0^n)$ as a $\Z$-module.  In other words $[\theta_\upsilon]$ generates $\St (\surface_0^n)$ as a $\MCG(\surface_0^n)$-module.

Suppose that $[\theta_\upsilon]$ is the trivial class in  $\St (\surface_0^n)$.  Then $\St (\surface_0^n)$ is the zero module.  Let $\Gamma < \MCG(\surface_0^n)$ be a torsion free finite index subgroup.  
Then by \cite[Theorem 4.1]{Harer1986} the cohomological dimension of $\Gamma$ is $n-3$.  We may use duality to show that for all $k$ and all $\Gamma$-modules $A$
\[
\HH^k(\Gamma;A) \cong \HH_{n-3-k}(\Gamma; A \otimes \St (\surface_0^n)) = \HH_{n-3-k}(\Gamma; A \otimes 0) =  \HH_{n-3-k}(\Gamma; 0) = 0.
\]
This implies that the cohomological dimension of $\Gamma$ is $0$.  Hence we have a contradiction.  It follows that $[\theta_\upsilon]$ is a non-trivial class in  $\St (\surface_0^n)$.
\end{proof}

\subsection{An essential sphere in the curve complex of the surface $\surface_0^n$}
\label{sec:ccsph}

Again we assume that $g=0$ and $n\geq 4$.   In  Proposition~\ref{prop:gen0} we have a non-trivial class in
\[
\St(\surface_0^n) = \frac{\CH_{n-3}(\AC)}{\CH_{n-3}(\AC_\infty) \oplus \partial \CH_{n-2}(\AC)}.
\]
We will now use the homotopy equivalence between $\AC_\infty$ and $\CC$ and the techniques of $\S$\ref{sec:simpmap} to get a $(n-4)$-sphere in $\CC$ representing this non-trivial class.  
After simplification our $(n-4)$-sphere will be the dual $D_{n-1}$ of the boundary of the associahedron $K_{n-1}$ (See Theorem~\ref{thm:assoc}).

The class of the filling arc system in Figure~\ref{fig:path} is represented by a single $(n-3)$-simplex $\theta_\upsilon \subset \AC (\surface_g^n)$ all of whose proper faces are contained in the arc complex at infinity.  The connecting homomorphism 
\[
\partial: \HH_{n-3}(\AC, \AC_\infty;\Z) \to \widetilde{\HH}_{n-4}(\AC_\infty;\Z) 
\]
sends the class $[\theta_\upsilon] \in \HH_{n-3}(\AC, \AC_\infty;\Z)$ to the class $[\theta_\upsilon^{n-4}] \in \widetilde{\HH}_{n-4}(\AC_\infty;\Z)$ where $\theta_\upsilon^{n-4}$ denotes the $(n-4)$-skeleton of the $(n-3)$-simplex $\theta_{\upsilon}$. 
Let $\bs: \AC_\infty \to \bs \AC_\infty$ be the barycentric subdivision map from (\ref{eq:bs}). 
Recall that the homotopy equivalence $\Phi : \bs \AC_\infty ( \surface_0^n ) \to \CC ( \surface_0^n )$ from (\ref{eqn:phi}) sends a non-filling arc system to the collection of boundary curves of a regular neighborhood of the arc system.
It follows that the homology class of the $(n-4)$-sphere $\Phi ( \bs \theta^{n-4}_\upsilon)$ is a $\MCG(\surface_0^n)$-module generator for $\widetilde{\HH}_{n-4}(\CC;\Z)$ and is non-trivial.

There are a number of properties of the  $(n-4)$-sphere $\Phi (\bs \theta^{n-4}_\upsilon)$ which are undesirable.  
Firstly, since $\Phi$ is defined on the barycentric subdivision of $\AC_\infty$ the sphere $\Phi ( \bs \theta^{n-4}_\upsilon)$ has $2^{(n-2)}-2$ vertices which is exponential in the number of marked points.  
Secondly, it is a subcomplex of the barycentric subdivision of the curve complex, whereas one would like to describe the sphere as a subcomplex of the curve complex itself.  
In order to address these issues we will replace the sphere $\Phi (\bs \theta^{n-4}_\upsilon)$ with a simpler one representing the same free homotopy class.  
In fact our simplified sphere $X \subset \CC$ described below will agree with $\Phi (\bs \theta^{n-4}_\upsilon)$ setwise but will have a much simpler structure as a simplicial complex.

Let $Y \subset \bs \AC_\infty$ be the flag complex whose vertices correspond to non-empty proper sub arc systems of the filling arc system $\upsilon$ in Figure~\ref{fig:path} and whose $k$-simplices correspond to increasing flags of such sub arc systems of length $k+1$.  
Notice that $Y = \bs \theta^{n-4}_\upsilon$ is the homeomorphic image of $\theta^{n-4}_\upsilon$ under the barycentric subdivision map $\bs : \AC_\infty \to \bs \AC_\infty$ from (\ref{eq:bs}) and hence topologically a sphere of dimension $n-4$.  
Let $i: Y \to \bs \AC_\infty$ be the inclusion map and let $\varphi = \Phi \circ i$ where $\Phi: \bs \AC_\infty \to \CC$ is as in (\ref{eqn:phi}).  Now we will apply the techniques of  $\S$\ref{sec:simpmap} to simplify $\varphi$.

Let $Y^0$ be the $0$-skeleton of $Y$.  Define the function $\rho^0:Y^0 \to \CC$ as follows. Let $\alpha \subsetneq \upsilon$ be a vertex of $Y$. 
Set $\alpha_\ast \subset \alpha$ to be the left-most connected component of the arc system $\alpha$ and $\rho^0(\alpha)$ to be the unique boundary curve of $\alpha_\ast$.  Notice that $\rho^0(\alpha) \in \supp \varphi(\alpha)$.  Hence by Corollary~\ref{cor:drop} the function $\rho^0$ is interpolable. Let
\[
\rho:Y \to \CC
\]
be its linear interpolation which by Corollary~\ref{cor:drop} is homotopic to $\varphi$.

Define the subcomplex $X = X(\surface_0^n) \subset \CC(\surface_0^n)$ as follows.
For each non-empty interval $J = \{ j, j+1, \cdots, m \}  \subsetneq \{1, 2, \cdots, n-2 \}$ let $x_J$ be the boundary curve of the arc system $\{y_j, y_{j+1}, \cdots, y_m \} \subset \upsilon$. Let
\[
X^0 = \{ x_J | \mbox{ $J \subsetneq \{ 1, 2, \cdots, n-2 \}$ is a non-empty interval} \}
\]
Let $X(\surface_0^n) = \Span X^0 \subset \CC(\surface_0^n)$.  The reader may now recognize the subcomplex $X$ as precisely the complex identified by Aramayona and Leininger in \cite[$\S$3]{AL}.  
For completeness of exposition we include Lemma~\ref{lem:Xassoc} below and in $\S$\ref{sec:al} below rectify the description of $X$ above with that given in \cite[$\S$3]{AL}.
\begin{lemma}[{\cite[Theorem 3.3]{AL}}]
\label{lem:Xassoc}
The complex $X$ and the simplicial complex $D_{n-1}$ (see Theorem~\ref{thm:assoc}) are isomorphic as simplicial complexes.
\end{lemma}
\begin{proof}
Let $\Lambda_{n-2}$ be the graph in  Figure~\ref{fig:pathonly}.  If $J,J' \subsetneq \{ 1, 2, \cdots, n-2 \}$ are two non-empty intervals then the boundary curves $x_{J}$ and $x_{J'}$ form a curve system exactly when $J$ and $J'$ are compatible tubes for the graph $\Lambda_{n-2}$.  
We caution the reader not to confuse the graph $\Lambda_{n-2}$ which has a {\em vertex} for each arc in the arc system $\upsilon$ and the underlying graph $G(\upsilon)$ for the arc system $\upsilon$ from $\S$\ref{sec:surfaces of genus 0} in which each arc of $\upsilon$ is an {\em edge}.

The linear interpolation $f:D_{n-1} \to X$ of the map sending the tube $J \in D_{n-1}^0$ of  $\Lambda_{n-2}$ to the curve $x_J \in X$ is a simplicial map by Remark~\ref{rem:simp} and its inverse is the linear interpolation of the function sending the curve $x_J$ to the maximal subinterval $J' \subset \{1, 2, \cdots, n-2 \}$ such that the arc system $\{y_j|j \in J'\}$ is entirely contained in the component of $\surface_0^n - x_J$ which does not contain the marked point $p_\upsilon$.
\end{proof}

\begin{lemma}
\label{lem:YtoX}
The map $\rho$ is a simplicial map, $\rho(Y) = X$ and $\rho: Y \to X$ is a homotopy equivalence. 
\end{lemma}
\begin{proof}
The map $\rho$ is a simplicial map onto its image by Remark~\ref{rem:simp} and the observation that for each vertex $\alpha$ of $Y$ we have
\[
|\supp \rho(\alpha)| = |\{ \rho^0(\alpha) \}| = 1.
\]

If $\alpha$ is a vertex of $Y$ then $\rho(\alpha)$ is a vertex of $X$.  The map $\rho$ is simplicial so $\rho(Y) \subset \Span X^0 = X$.

Observe that $Y$ is topologically an $(n-4)$-sphere and by Lemma~\ref{lem:Xassoc} that $X$ is topologically an $(n-4)$-sphere.  
Hence the map $\rho : Y \to X$ will be a homotopy equivalence if and only if it is a degree $\pm 1$ map.  
Let $J = \{1,2, \cdots, n-3 \}$ then  $\rho^{-1}( x_{J}) = \{\alpha\}$ where $\alpha = \{y_1, y_2, \cdots, y_{n-3} \}$.  
Thus the degree of $\rho$ is equal to the degree of the induced map $\rho^\ast : \link(\alpha,Y) \to \link( x_{J},X)$ where $\link(w,W)$ denotes the link of the vertex $w$ in the complex $W$.  
The map $\rho^\ast$ is exactly the map $\rho$ for the surface with one less puncture.  
For the base case surface $\surface_0^4$ the map $\rho:Y \to X$ is a homeomorphism of $0$-spheres which has degree $\pm 1$.  
Hence, inductively, for all $n \geq 4$ the map $\rho:Y \to X$ has degree $\pm 1$ and is therefore a homotopy equivalence.  It follows that $\rho : Y \to X$ is surjective.
\end{proof}

We have arrived at the following simplification of our $\MCG(\surface_0^n)$-module generator for $\widetilde{\HH}_{n-4}(\CC;\Z)$.

\begin{proposition}
\label{prop:simpsphere}
Assume $n \geq 4$. Let $X = X(\surface_0^n) \subset \CC(\surface_0^n)$ be the subcomplex of all simplices of $\CC$ whose vertices are boundary curves of connected sub arc systems of the arc system $\upsilon$ in Figure~\ref{fig:path}.  
Then as a simplicial complex $X$ is the dual $D_{n-1}$ of the boundary of the associahedron $K_{n-1}$ (see Theorem~\ref{thm:assoc}) and with a proper choice of orientation the class $[X] \in  \widetilde{\HH}_{n-4}(\CC;\Z)$ is the same as that of $\theta_\upsilon$ from Proposition~\ref{prop:gen0} and hence is a non-trivial generator for $\widetilde{\HH}_{n-4}(\CC;\Z)$ as a $\MCG(\surface_0^n)$-module.
\end{proposition}

\begin{proof}
Lemma~\ref{lem:Xassoc} shows that $X$ and $D_{n-1}$ are isomorphic simplicial complexes.
The map $\varphi: Y \to \CC$ above represents a $\MCG(\surface_0^n)$-module generator for $\widetilde{\HH}_{n-4}(\CC;\Z)$.  As we noted $\rho: Y \to \CC$ is homotopic to $\varphi$ and thus represents the same class.  Let $i_X: X \to \CC$ be the inclusion map.  By Lemma~\ref{lem:YtoX} the map $\rho: Y \to X$ has a homotopy inverse $h: X \to Y$.  Thus $i_X$ is homotopic to $\rho \circ h$ which represents the same class as $\varphi$ in $\widetilde{\HH}_{n-4}(\CC;\Z)$.
\end{proof}

\subsection{The proof of Theorem~\ref{thm:same}}
\label{sec:al}

In \cite{AL} Aramayona and Leininger show that certain finite subcomplexes $\mathfrak{W} \subset \CC$ of the curve complex are {\em rigid sets}.  
That is, any injective simplicial map $f:\mathfrak{W} \to \CC$ extends to a simplicial injection $F:\CC \to \CC$ and hence is induced by a mapping class.

In the case that $\surface = \surface_0^n$ their finite rigid set $\mathfrak{X} = \mathfrak{X}(\surface_0^n)$ is a topological sphere of dimension $n-4$ and they ask \cite[Question 2]{AL} if the class $[\mathfrak{X}] \in \widetilde{\HH}_{n-4}(\CC (\surface_0^n))$ is non-trivial.  
The construction of the set $\mathfrak{X} \subset \CC (\surface_0^n)$ is as follows:  Let $G_n$ be a regular $n$-gon and let $\surface_0^n$ be the double of $G_n$ along its boundary with each vertex of $G_n$ giving a marked point.  
For each line segment $w$ joining two nonadjacent sides of $G_n$ the double $c_w$ of $w$ will be an essential curve in $\surface_0^n$.  
Let $\mathfrak{X}^0$ be the set of all $\binom{n}{2} - n$ essential curves in $\surface_0^n$ which may be constructed in this manner.  Define $\mathfrak{X} = \Span \mathfrak{X}^0$ to be the union of all simplices of $\CC(\surface_0^n)$ whose vertices all lie in the set $\mathfrak{X}^0$.

The reader is referred to $\S$\ref{sec:intro} of this paper for the statement of Theorem~\ref{thm:same}.  Using Proposition~\ref{prop:simpsphere},  we are now ready to give the proof.

\begin{proof}[Proof of Theorem~\ref{thm:same}]
Assume that $n \geq 5$.  Arrange the arc system in Figure~\ref{fig:path} so that it is entirely contained in the image of the boundary of the $n$-gon $G_n$ in $\surface_0^n$.  
Arrange the arcs in Figure~\ref{fig:path} so that they coincide with all but two adjacent edges of $G_n$.  Let $p_\upsilon$ be the marked point of $\surface_0^n$ which is not incident with any arcs.  
If $w$ is a line segment joining two nonadjacent sides of $G_n$ then its double $c_w$ will be the boundary curve for the maximal connected sub arc system of the arc system  in Figure~\ref{fig:path} which is disjoint from $c_w$ and which is in the component of $\surface_0^n  -  c_w$ which does not contain the marked point $p_\upsilon$.  
This pairing identifies the vertices of $\mathfrak{X}$ and $X$.  Since both subcomplexes are the full subcomplexes of $\CC (\surface_0^n)$ spanned by their vertices the subcomplexes must coincide.
\end{proof}
We address the case of the surface $\surface_0^4$  in $\S$\ref{SS:sporadic} below.

\section{Essential spheres and finite rigid sets for the surfaces $\surface^0_g$ and $\surface^1_g$}
\label{sec:genusg}

We now turn our attention to the surface $\surface_g^n$ where $g \geq 1$ and $n \in \{ 0,1 \}$.  Our main interest is in the case $n=0$, but in the process of investigating that we will need to study the case $n=1$.  
To get started, we review what was done in \cite{Broaddus} to identify the homologically non-trivial sphere in $\CC(\surface_g^n)$ described there.  
After that, we use the techniques of $\S$\ref{sec:simpmap} to radically simplify this sphere.  Finally in $\S$\ref{sec:gcomp} we compare this simplified sphere to the finite rigid sets for $\surface_g^n$ given in \cite{AL}.

\subsection{An essential sphere in the curve complex of a surface of genus $g>0$}
\label{sec:ggen}

Consider the surface of genus $g$ with $1$ marked point depicted in Figure~\ref{fig:salgen}.  Let  $\theta_\zeta \subset \AC(\surface_g^1)$ be the $(2g-1)$-simplex corresponding to the filling arc system $\zeta$ given in Figure~\ref{fig:salgen}. 
In \cite{Broaddus} it is shown that the class $[\theta_\zeta]$ generates 
\[
\St(\surface_g^1) \cong \frac{\CH_{2g-1}(\AC)}{\CH_{2g-1}(\AC_\infty) \oplus \partial \CH_{2g}(\AC)}
\]
as a $\MCG(\surface_g^1)$-module.
\begin{figure}[htbp]
\begin{center}
\includegraphics{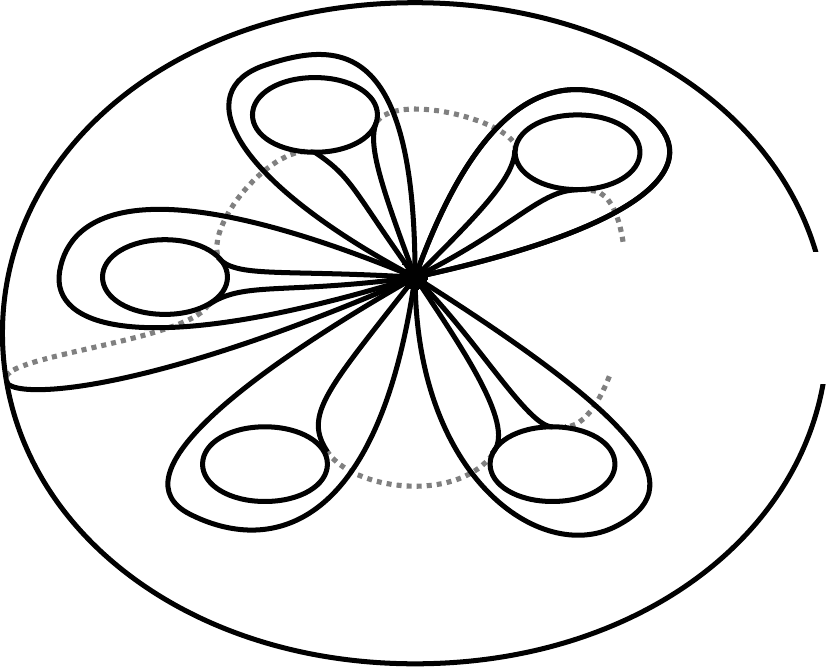}
\put(-64,97){\rotatebox{-5}{\Huge$\vdots$}}
\put(-8,97){\rotatebox{1}{\Huge$\vdots$}}
\put(-219,72){$z_1$}
\put(-215,135){$z_2$}
\put(-170,118){$z_3$}
\put(-187,160){$z_4$}
\put(-145,156){$z_5$}
\put(-48,156){$z_6$}
\put(-78,142){$z_7$}
\put(-92,60){$z_{2g-3}$}
\put(-80,30){$z_{2g-2}$}
\put(-173,57){$z_{2g-1}$}
\put(-190,34){$z_{2g}$}
\end{center}
\caption{A filling arc system $\zeta = \{ z_1, z_2, \cdots, z_{2g} \}$ with  $2g$ arcs in the surface $\surface_g^1$ whose class generates $\St (\surface_g^1)$ as a $\MCG(\surface_g^1)$-module.\label{fig:salgen}}
\end{figure}
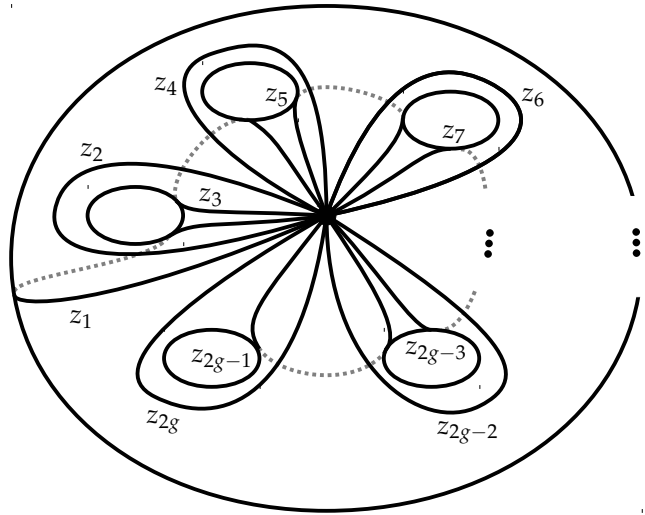
As was the case in Proposition~\ref{prop:gen0} this class in $\St (\surface_g^1)$ must be non-trivial.

Now consider the surface $\surface_g^0$ without marked points.  Here we have no arc complex, but by a result of Harer \cite[Lemma 3.6]{Harer1986} following from more general work by Kent, Leininger and Schleimer \cite[Corollary 1.1]{KLS} the map
\begin{equation}
\label{eq:psi}
\Psi : \CC(\surface_g^1) \to \CC(\surface_g^0)
\end{equation}
which ``forgets'' the marked point is a homotopy equivalence.  Hence we may identify $\St (\surface_g^1)$ and $\St (\surface_g^0)$ via the induced map
\[
\Psi_\ast :  \widetilde{\HH}_{2g-2}(\CC(\surface_g^1);\Z) \to \widetilde{\HH}_{2g-2}(\CC(\surface_g^0);\Z)
\]
for all $g \geq 1$.

To continue,  suppose that $g \geq 1$ and $n=1$. 
Let $Z \subset \bs \AC_\infty(\surface_g^1) $ be the flag complex whose vertices correspond to non-empty proper sub arc systems of the filling arc system $\zeta$ in Figure~\ref{fig:salgen} and whose $k$-simplices correspond to increasing flags of such sub arc systems of length $k+1$.
The simplicial complex $Z$ is the barycentric subdivision of the boundary of a $(2g-1)$-simplex and hence topologically a sphere of dimension $2g-2$.  Again we let $i: Z \to \bs \AC_\infty$ be the inclusion map and set $\varphi = \Phi \circ i$ where $\Phi: \bs \AC_\infty \to \CC$ is as in (\ref{eqn:phi}).

The {\em sequential components} of a sub arc system $\alpha$ of the arc system $\zeta$ in Figure~\ref{fig:salgen} will be the maximal subsets of $\alpha$ of the form $\{z_j, z_{j+1}, \cdots, z_{k} \}$ where $1 \leq j \leq k \leq 2g$.
Note that $\alpha$ is the disjoint union of its sequential components.  
If $\alpha_1,\alpha_2 \subset \alpha$ are sequential components of $\alpha$ then we say that $\alpha_1$ {\em precedes} $\alpha_2$ if for all $z_j \in \alpha_1$ and all $z_k \in \alpha_2$ we have $j < k$.   
Under this order we will refer to the {\em first}, {\em second}, {\em last}, {\em etc.} sequential component of $\alpha$.  
The order on the sequential components of a sub arc system $\alpha \subsetneq \zeta$ should not be confused with the containment partial order on all arc systems under which the sequential components of $\alpha$ are incomparable. 
A sequential component will be called {\em odd} if it has an odd number of elements and {\em even} if it has an even number of elements.

The first step in the simplification of $\varphi : Z \to\CC(\surface_{g}^1)$ will be to define a map $\rho : Z \to\CC(\surface_{g}^1)$ which will prove to be homotopic to $\varphi$. (See Lemma~\ref{lem:hom} below). Let $Z^0$ be the $0$-skeleton of $Z$.  
Define the function $\rho^0 : Z^0 \to \CC(\surface_{g}^1)$ as follows. Let $\alpha \in Z^0$ be a proper sub arc system of $\zeta$ and let $\alpha^\star = \{ z_j, z_{j+1}, \cdots, z_{k} \}$ be the first sequential component of $\alpha$.  
Let $N(\alpha^\star)$ be a closed regular neighborhood of $\bigcup \alpha^\star$ in $\surface_g^1$.  If $\alpha^\star$ is an odd sequential component set $N'(\alpha) = N(\alpha^\star)$ and let $\rho^0(\alpha)$ be the boundary component of $N'(\alpha)$ which is to the right when following $z_j$ in the clockwise direction. (See sketch (i) of Figure~\ref{fig:oesc}). 
\begin{figure}[htbp]
\begin{center}
\begin{tabular}{cc}
\raisebox{.4cm}{
\includegraphics{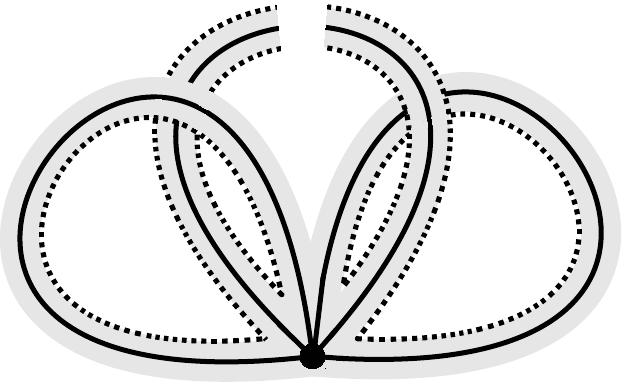}
\put(-99,99){$\cdots$}
\put(-170,77){$z_j$}
\put(-140,102){$z_{j+1}$}
\put(-60,102){$z_{j+2m-1}$}
\put(-34,86){$z_{j+2m}$}
\put(-163,35){$\rho^0(\alpha)$}
}
&
\includegraphics{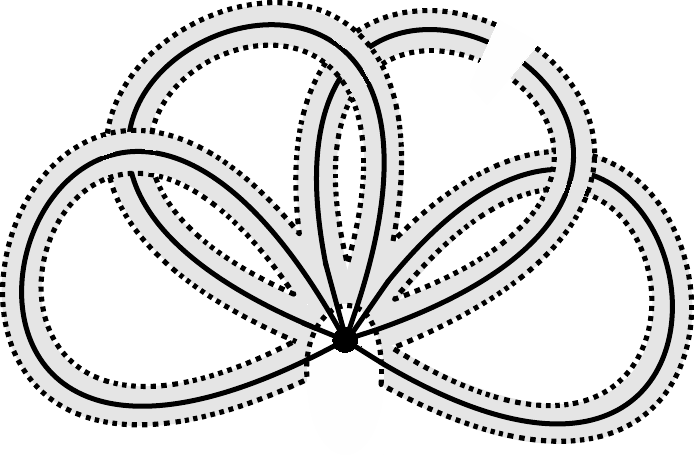}
\put(-64.5,111.5){\rotatebox{60}{$\vdots$}}
\put(-192,87){$z_j$}
\put(-178,115){$z_{j+1}$}
\put(-135,92){$z_{j+2}$}
\put(-30,102){$z_{j+2m}$}
\put(-54,26){$z_{j+2m+1}$}
\put(-182,35){$\rho^0(\alpha)$}
\\
(i) $\alpha^\star$ is an odd sequential component
&
(ii) $\alpha^\star$ is an even sequential component
\\
\end{tabular}
\end{center}
\caption{Let $\alpha^\star$ be the first sequential component of the arc system $\alpha$.  
Sketch (i) depicts the case where $\alpha^\star = \{ z_j, z_{j+1}, \cdots, z_{j+2m} \}$ is odd and sketch (ii) depicts the case where $\alpha^\star = \{ z_j, z_{j+1}, \cdots, z_{j+2m+1} \}$ is even.  
In both cases $N'(\alpha)$ is the grey region and $\rho^0(\alpha)$ is the dotted line.\label{fig:oesc}}
\end{figure}
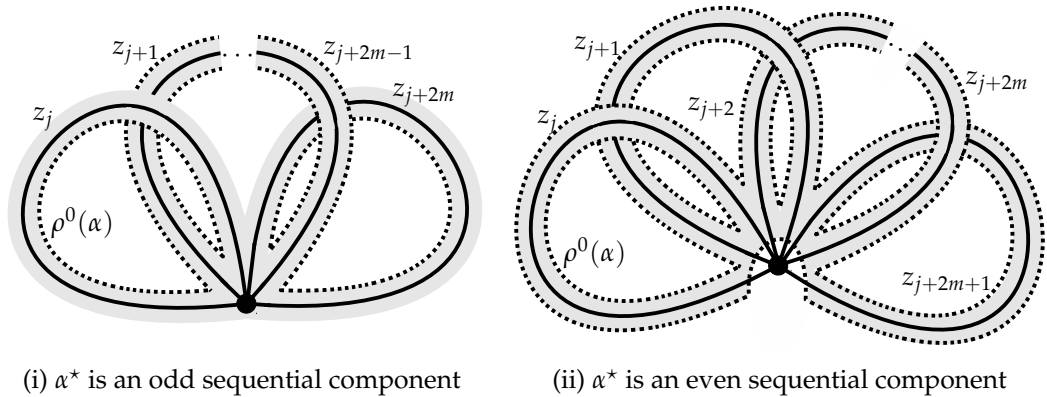

If $\alpha^\star = \{ z_j, z_{j+1}, \cdots, z_{k} \}$ is an even sequential component let $N'(\alpha) = N(\alpha^\star) - \inte D$ where $D$ is a closed regular neighborhood  of the marked point in $\surface_{g}^1$ whose intersection with the boundary curve  
of $N(\alpha^\star)$ is a single line segment and which contains $N(\alpha^\star) \cap z_i$ for each $i$ satisfying $i \leq j-2$ or $i \geq k+2$. Let $\rho^0(\alpha)$ be the unique boundary component of $N'(\alpha)$. (See sketch (ii) of Figure~\ref{fig:oesc}).

\begin{lemma}
\label{lem:self}
If $\alpha \subset \beta$ are two proper sub arc systems of the arc system $\zeta$ in Figure~\ref{fig:salgen} then $\{ \rho^0(\alpha),\rho^0(\beta) \}$ is a curve system.
\end{lemma}
\begin{proof}
Let $\alpha \subset \beta$ be two proper sub arc systems of the arc system $\zeta$ in Figure~\ref{fig:salgen} with respective first sequential components $\alpha^\star$ and $\beta^\star$.
Then either $\alpha^\star \subset \beta^\star$ or $\alpha^\star \cup \beta^\star$ has two sequential components with $\beta^\star$ preceding $\alpha^\star$.

Suppose $\alpha^\star \subset \beta^\star$.  If $\beta^\star$ is an odd sequential component or both $\alpha^\star$ and $\beta^\star$ are even sequential components then we may choose the closed regions $N'(\alpha)$ and $N'(\beta)$ so that  $N'(\alpha) \subset \inte N'(\beta)$. 
Hence the curves $\rho^0(\alpha)$ and $\rho^0(\beta)$ are disjoint and form a curve system.

Again suppose $\alpha^\star \subset \beta^\star$. If $\beta^\star$ is an even sequential component and $\alpha^\star$ is an odd sequential component then we may choose the closed region $N'(\alpha)$ and a closed regular neighborhood $N(\beta^\star)$ of $\bigcup \beta^\star$ so that $N'(\alpha) \subset \inte N(\beta^\star)$. 
We may choose the closed disk neighborhood $D$ of the marked point so that it is disjoint from the boundary curve $\rho^0(\alpha)$ of $N'(\alpha)$ and set $N'(\beta) = N(\beta^\ast) - \inte D$.  We again see that the curves $\rho^0(\alpha)$ and $\rho^0(\beta)$ are disjoint and form a curve system.

Suppose $\alpha^\star \cup \beta^\star$ has two sequential components.  If at least one of $\alpha^\star$ and $\beta^\star$ is an even sequential component then we may arrange for the regions $N'(\alpha)$ and $N'(\beta)$ to be disjoint.  
Thus the curves $\rho^0(\alpha)$ and $\rho^0(\beta)$ are disjoint and form a curve system.
If both $\alpha^\star$ and $\beta^\star$ are odd sequential components then we may choose $N'(\alpha)$ and $N'(\beta)$ so that their intersection is a bigon containing the marked point which is disjoint from $\rho^0(\alpha)$ and $\rho^0(\beta)$.  Again we see that the curves $\rho^0(\alpha)$ and $\rho^0(\beta)$ are disjoint.
\end{proof}

Recall that a set of curves forms a curve system if and only if every subset of size two is a curve system.  Let $\sigma$ be a simplex of $Z$ with vertex set $\sigma^0$.  
Then every subset of $\rho^0(\sigma^0)$ of size two is of the form $\{ \rho^0(\alpha),\rho^0(\beta) \}$ for some arc systems $\alpha, \beta$ with $\alpha \subsetneq \beta \subsetneq \zeta$.  
By Lemma~\ref{lem:self} all such sets of size two are curve systems.  It follows that $\rho^0(\sigma^0)$ is a curve system and that $\rho^0$ is interpolable.
Let
\[
\rho: Z \to \CC(\surface_g^1)
\]
be its linear interpolation.

In Lemma~\ref{lem:hom} below we will construct an explicit homotopy between $\varphi: Z \to \CC(\surface_g^1)$ and $\rho: Z \to \CC(\surface_g^1)$.  
In order to do this we will need Lemma~\ref{lem:lower} below.

\begin{lemma}
\label{lem:lower}
If $\alpha \subset \beta$ are two proper sub arc systems of the arc system $\zeta$ in Figure~\ref{fig:salgen} then
\[
\supp \rho(\alpha) \cup \supp \varphi(\beta)
\]
is a curve system.
\end{lemma}
\begin{proof}
Let $\alpha \subset \beta$ be two proper sub arc systems of the arc system $\zeta$.  Then we may choose a regular neighborhood $N(\beta)$ of $\bigcup \beta$ in $\surface_g^1$ and a closed region $N'(\alpha)$ so that $N'(\alpha) \subset \inte N(\beta)$.  
The curves in $\supp \varphi(\beta)$ may be chosen to be the pairwise disjoint boundary curves of $N(\beta)$ and the single curve in $\supp \rho(\alpha)$ may be taken to be the boundary curve of $N'(\alpha)$.  
Therefore the curves in $\supp \rho(\alpha) \cup \supp \varphi(\beta)$ are pairwise disjoint and form a curve system.
 \end{proof}

Now we have the tools to construct an explicit homotopy equivalence between $\varphi$ and $\rho$.
\begin{lemma}
\label{lem:hom}
The maps $\varphi: Z \to \CC(\surface_g^1)$ and $\rho: Z \to \CC(\surface_g^1)$ are homotopic.
\end{lemma}
\begin{proof}
Let $\sigma$ be a simplex of $Z$ with vertex set $\sigma^0 = \{\alpha_0, \alpha_1, \cdots, \alpha_k \}$ satisfying $\alpha_0 \subsetneq \alpha_1 \subsetneq \cdots \subsetneq \alpha_k$.
Consider the simplicial complex $P = \sigma \times I$ which is a prism triangulated by $k+1$ simplices of dimension $k+1$ with vertex sets:
\[
\begin{split}
P_0 &  = \{ (\alpha_0,0), (\alpha_0,1), (\alpha_1,1), \cdots, (\alpha_k,1) \}\\
P_1 &  = \{ (\alpha_0,0), (\alpha_1,0),(\alpha_1,1) \cdots, (\alpha_k,1) \}\\
& \vdots \\
P_k &  = \{ (\alpha_0,0), (\alpha_1,0), \cdots, (\alpha_k,0), (\alpha_k,1) \}\\
\end{split}
\]
The $0$-skeleton of $P$ is the set $\sigma^0 \times \{0,1\}$.  Let $F_\sigma^0: \sigma^0 \times \{0,1\} \to  \CC(\surface_g^1)$ be the function
\[
F_\sigma^0(\alpha_i, \varepsilon) = \left \{\begin{array}{ll} \rho(\alpha_i), & \varepsilon =0 \\ \varphi(\alpha_i), & \varepsilon=1\end{array} \right. .
\]

We wish to show that $F_\sigma^0$ is interpolable.  Each maximal simplex of $P = \sigma \times I$ has vertex set $P_i$ for some $i$.  
Thus we must show that $\supp (F_\sigma^0(P_i))$ is a curve system for each $i$.  Recall that a set of curves is a curve system if and only if each subset of size two is a curve system.  
A subset of size two of the set $\supp (F_\sigma^0(P_i))$ must be a subset of a set of the form  $\supp F_\sigma^0(\alpha, 0) \cup \supp F_\sigma^0(\beta, 0)$, a set of the form $\supp F_\sigma^0(\alpha, 0) \cup \supp F_\sigma^0(\beta, 1)$, or a set of the form $ \supp F_\sigma^0(\alpha, 1)  \cup \supp F_\sigma^0(\beta, 1)$
for some curve systems $\alpha \subset \beta \subsetneq \zeta$.  The set
\[
\supp F_\sigma^0(\alpha, 0) \cup \supp F_\sigma^0(\beta, 0)  =  \{ \rho(\alpha) ,  \rho(\beta)\}
\]
is a curve system by Lemma~\ref{lem:self}.  The set
\[
\supp F_\sigma^0(\alpha, 0) \cup \supp F_\sigma^0(\beta, 1) = \supp \rho(\alpha) \cup \supp \varphi(\beta)
\]
is a curve system by Lemma~\ref{lem:lower}.  Finally the set
\[
\supp F_\sigma^0(\alpha, 1) \cup \supp F_\sigma^0(\beta, 1) = \supp \varphi(\alpha) \cup \supp \varphi(\beta)
\]
is a curve system since it consists of boundary curves of two closed subsurfaces $N(\alpha)$ and $N(\beta)$ of $\surface_g^1$ with $N(\alpha) \subset \inte N(\beta)$.
Hence $F_\sigma^0$ is interpolable.  Let $F_\sigma: \sigma \times I \to  \CC(\surface_g^1)$ be its linear interpolation.  By the gluing lemma these $F_\sigma$'s assemble to a homotopy $F: Z \times I  \to \CC(\surface_g^1)$  from $\rho$ to $\varphi$.
\end{proof}

We now describe a subcomplex $X(\surface_g^1) \subset \CC(\surface_g^1)$ which will prove to be the image of the map $\rho$.  For $1 \leq i \leq 2g$ let $z_i' = \rho(z_i)$ (see Figure~\ref{fig:chain}).   \begin{figure}[htbp]
\begin{center}
\includegraphics{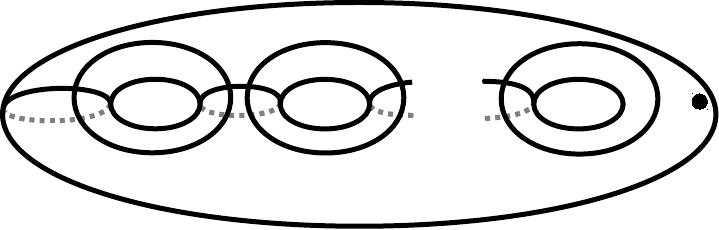}
\put(-84,37){$\cdots$}
\put(-217,32){$z_1'$}
\put(-148,52){$z_2'$}
\put(-122,34){$z_3'$}
\put(-98,52){$z_4'$}
\put(-72,52){$z_{2g}'$}
\end{center}
\caption{\label{fig:chain}}
\end{figure}
For each interval $J  = \{j, j+1, \cdots, m \} \subsetneq \{ 1 , 2, \cdots, 2g \}$ let $N_J$ be a closed regular neighborhood of the set $z_j' \cup z_{j+1}' \cup \cdots \cup z_m'$.  
If $|J|$ is even then $N_J$ has one boundary component.  Let $x_J$ be that boundary component.  
If $|J|$ is odd then $N_J$ has two boundary components.  In this case let $x_J$ be the boundary component of $N_J$ in the ``back half'' (if $j$ is even) or ``top half'' (if $j$ is odd) of the surface.  Let $X^0$ be the set
\begin{equation}
\label{eq:xdef}
X^0 = \{ x_J | \mbox{ $J \subsetneq  \{ 1 , 2, \cdots, 2g \}$ is an interval} \}
\end{equation}
and define $X(\surface_g^1) = \Span X^0$ to be the subcomplex of all simplices of $\CC(\surface_g^1)$  whose vertices are in $X^0$.  
Let $X(\surface_g^0) = \Psi(X(\surface_g^1) )$ be the image of $X(\surface_g^1)$ under the map $\Psi:\CC(\surface_g^1) \to \CC(\surface_g^0)$ from (\ref{eq:psi}) above which ``forgets'' the marked point.
\begin{lemma}
\label{lem:gassoc}
Assume $g \geq 1$ and $n \in \{0,1\}$.  The simplicial complex $X(\surface_g^n)$ is isomorphic as a simplicial complex to $D_{2g+1}$ (see Theorem~\ref{thm:assoc}).
\end{lemma}
\begin{proof}
For intervals $J_1, J_2 \subsetneq  \{ 1 , 2, \cdots, 2g \}$ the curves $x_{J_1}$ and  $x_{J_2}$ are disjoint precisely when the tubes $J_1$ and $J_2$ for the graph $\Lambda_{2g}$ in Figure~\ref{fig:pathonly} are compatible.  
Let $D_{2g+1}^0$ be the $0$-skeleton of the dual $D_{2g+1}$ of the boundary of the associahedron $K_{2g+1}$.  
The bijection $f^0:X^0 \to D_{2g+1}^0$ given by $f^0(x_J) =  J$ and its inverse are interpolable maps showing that $X(\surface_g^1)$ and $D_{2g+1}$ are isomorphic as simplicial complexes. 
The map  $\Psi |_{X(\surface_g^1)}: X(\surface_g^1) \to X(\surface_g^0)$ is an isomorphism of simplicial complexes.
\end{proof}
\begin{lemma}
\label{lem:ZtoX}
Assume $g \geq 1$ and $n \in \{0,1\}$.  The maps $\rho: Z \to \CC(\surface_g^1)$ and   $\Psi \circ \rho: Z \to \CC(\surface_g^0)$ are simplicial maps.  The images of $\rho$ and  $\Psi \circ \rho$ are $X(\surface_g^1)$ and $X(\surface_g^0)$ respectively.  Furthermore $\rho: Z \to X(\surface_g^1)$ and   $\Psi \circ \rho: Z \to X(\surface_g^0)$ are homotopy equivalences.  
\end{lemma}

\begin{proof}
The map $\rho: Z \to \CC(\surface_g^1)$ is a simplicial map onto its image by Remark~\ref{rem:simp} and the observation that for each vertex $\alpha$ of $Z$ we have
\[
|\supp \rho(\alpha)| = 1
\]
The map $\Psi$ a simplicial map so the composition $\Psi \circ \rho$ is simplicial.

The curves in the set $X^0$ given in (\ref{eq:xdef}) are precisely the curves in the image of $\rho$.  
Thus if $\alpha$ is a vertex of $Z$ then $\rho(\alpha)$ is a vertex of $X(\surface_g^1)$ and $\Psi \circ \rho(\alpha)$ is a vertex of  $X(\surface_g^0)$.  
The maps $\rho$ and $\Psi \circ \rho$ are  simplicial so $\rho(Z) \subset X(\surface_g^1)$ and $\Psi \circ \rho(Z) \subset X(\surface_g^0)$.

Observe that $Z$ is topologically a $(2g-2)$-sphere and by Lemma~\ref{lem:gassoc} that $X(\surface_g^1)$ is topologically an $(2g-2)$-sphere.  
Moreover, as maps between abstract simplicial complexes the map $\rho:Z  \to X(\surface_g^1)$ agrees with the map $\rho:Y  \to X(\surface_0^{2g+2})$ from Lemma~\ref{lem:YtoX}. 
Hence the map $\rho : Z \to X(\surface_g^1)$ is a homotopy equivalence and has degree $\pm 1$.   It follows that  $\rho : Z \to X(\surface_g^1)$ is surjective.

The restriction $\Psi |_{X(\surface_g^1)}: X(\surface_g^1) \to X(\surface_g^0)$ is an isomorphism of simplicial complexes and hence a degree $\pm 1$ map.  It follows that  $\Psi \circ \rho: Z \to X(\surface_g^0)$ has degree $\pm 1$. It is therefore a surjective simplicial homotopy equivalence.
\end{proof}

\begin{proposition}
\label{prop:geng}
Let $g \geq 1$ and  $n \in \{ 0,1\}$.  Let $X^0$ be the set of essential curves in $\surface_g^n$  given in (\ref{eq:xdef}) and $X(\surface_g^n) \subset \CC(\surface_g^n)$ be the full subcomplex of $\CC(\surface_g^n)$ generated by $X^0$.
Then as a simplicial complex $X(\surface_g^n)$ is isomorphic to the dual $D_{2g+1}$ of the boundary of the associahedron $K_{2g+1}$ (see Theorem~\ref{thm:assoc}).  
Furthermore $X(\surface_g^n)$ is topologically a sphere of dimension $2g-2$ and, choosing an orientation, the class $[X(\surface_g^n)] \in \widetilde{\HH}_{2g-2}(\CC(\surface_g^n);\Z)$ is non-trivial and generates $\widetilde{\HH}_{2g-2}(\CC(\surface_g^n);\Z)$ as  a $\MCG(\surface_g^n)$-module.
\end{proposition}

\begin{proof}
Lemma~\ref{lem:gassoc} shows that $X(\surface_g^1)$, $X(\surface_g^0)$ and $D_{2g+1}$ are isomorphic simplicial complexes.
The map $\varphi: Z \to \CC(\surface_g^1)$ (respectively $\Psi \circ \varphi: Z \to \CC(\surface_g^0)$) above represents a $\MCG(\surface_g^1)$-module (respectively  $\MCG(\surface_g^0)$-module) generator for $\widetilde{\HH}_{2g-2}(\CC(\surface_g^1);\Z)$ (respectively  $\widetilde{\HH}_{2g-2}(\CC(\surface_g^0);\Z)$).  
As we noted $\rho$ is homotopic to $\varphi$ so  $\Psi \circ \rho$ is homotopic to $\Psi \circ \varphi$.  Let $i_{X(\surface_g^1)}: X(\surface_g^1) \to \CC(\surface_g^1)$ and $i_{X(\surface_g^0)}: X(\surface_g^0) \to \CC(\surface_g^0)$  be the inclusion maps.  
By Lemma~\ref{lem:ZtoX} the maps $\rho: Z \to X(\surface_g^1)$ and $\Psi \circ \rho: Z \to X(\surface_g^0)$ have homotopy inverses $h_1: X(\surface_g^1) \to Z$ and $h_0: X(\surface_g^0) \to Z$ respectively.  
Thus $i_{X(\surface_g^1)}$ and  $i_{X(\surface_g^0)}$ are  homotopic to $\rho \circ h_1$ and $\Psi \circ \rho \circ h_0$ respectively which represent the same classes as $ \varphi$ in $\widetilde{\HH}_{n-4}(\CC(\surface_g^1);\Z)$ and $\Psi \circ \varphi$ in $\widetilde{\HH}_{n-4}(\CC(\surface_g^0);\Z)$ respectively.
\end{proof}

\subsection{The proof of  Theorem~\ref{thm:notiff}}
\label{sec:gcomp}

Again let $g \geq 1$ and $n \in \{0,1\}$.  In \cite{AL} a finite rigid set $\mathfrak{X}(\surface_g^n) \subset \CC(\surface_g^n)$ is given.  Our essential sphere $X(\surface_g^n)$ is a subset of $\mathfrak{X}(\surface_g^n)$.  Given the surprising coincidence of our essential sphere and the finite rigid set of Aramayona and Leininger for the surface $\surface_0^n$ one might be tempted to conjecture that a subcomplex of the curve complex is essential if and only if it is finitely rigid.   As suggested by Theorem~\ref{thm:notiff} this is not the case. See $\S$\ref{sec:intro} of this paper for the statement of Theorem~\ref{thm:notiff}.

\begin{proof}[Proof of Theorem~\ref{thm:notiff}]
Suppose that $g \geq 3$ and $n \in \{0,1\}$.
Let $X^0$ be the set of curves described in (\ref{eq:xdef}).  
Let $N_{\{1,2,3\}}$ be a closed regular neighborhood of $z_1' \cup z_2' \cup z_3'$ where the curves $z_1', z_2', z_3'$ are as in Figure~\ref{fig:chain} above.  
Then $N_{\{1,2,3\}}$ has two boundary components.  One boundary component $x_{\{1,2,3\}}$ of $N_{\{1,2,3\}}$ is in $X^0$ the other $x_{\{1,2,3\}}'$ is not.  Let $f^0: X^0 \to \CC(\surface_g^n)$ be the function
\[
f^0(y) = \left\{ \begin{array}{ll} y, & y \neq x_{\{1,2,3\}} \\ x_{\{1,2,3\}}', & y =  x_{\{1,2,3\}} \end{array} \right. .
\]
Every curve in $X^0$ that is disjoint from $x_{\{1,2,3\}}$ is also disjoint from $x_{\{1,2,3\}}'$. Thus $f^0$ is interpolable with linear interpolation $f:X \to \CC(\surface_g^n)$.  
The map $f$ is simplicial since $|\supp f(x)| = 1$ for all $x \in X^0$.  
The map $f$ is an injection since if $x,y \in X^0$ are distinct disjoint curves then $f(x)$ and $f(y)$ are distinct disjoint curves. 
The map $f$ cannot be induced by an extended mapping class since $x_{\{1,2,3\}}$  has non-empty intersection with $x_{\{3,4,5\}}$ but $f(x_{\{1,2,3\}}) = x_{\{1,2,3\}}'$ and $f(x_{\{3,4,5\}}) = x_{\{3,4,5\}}$ are disjoint.

If $g=2$ and $n=1$ then $X \subsetneq \mathfrak{X}$. Again, we claim that $X$ is not a finite rigid set.  Let $X^0$ be the set of curves described in (\ref{eq:xdef}). 
 Let $N_{\{1,2\}}$ be a closed regular neighborhood of $z_1' \cup z_2'$ where the curves $z_1'$ and $z_2' $ are as in Figure~\ref{fig:chain} above.  
 Then $N_{\{1,2\}}$ has a single boundary component $x_{\{1,2\}}$.  
 Fix a simple path $p$ from $x_{\{1,2\}}$ to the marked point which is disjoint from all other curves in $X^0$.  Let $x_{\{1,2\}}'$ be the boundary component of a regular neighborhood of $x_{\{1,2\}} \cup p$ which is not isotopic to $x_{\{1,2\}}$.   Let $f^0: X^0 \to \CC(\surface_g^n)$ be the function
\[
f^0(y) = \left\{ \begin{array}{ll} y, & y \neq x_{\{1,2\}} \\ x_{\{1,2\}}', & y =  x_{\{1,2\}} \end{array} \right. .
\]
The function $f^0$ preserves disjointness of curves.  Thus $f^0$ is interpolable with continuous linear interpolation $f:X \to \CC(\surface_g^n)$.  
The map $f$ is simplicial since $|\supp f(x)| = 1$ for all $x \in X^0$.  The map $f$ is an injection since if $x,y \in X^0$ are distinct disjoint curves then $f(x)$ and $f(y)$ are distinct disjoint curves.

We claim that $f$ is not induced by an extended mapping class $h \in \MCG^\pm(\surface_g^n)$.  Note that $f$ is the identity on the chain curves $z_1', z_2', z_3', z_4'$.  
Suppose that the mapping class $h \in \MCG^\pm(\surface_g^n)$ extends $f$.  
Then without loss of generality we may represent the class of $h$  by a map which induces an isometry on the regular $12$-gon $G_{12}$ one gets by cutting along all of the chain curves.    
There are 4 possibilities for $h$ generated by the  orientation reversing reflection about the plane through the even curves  $z_2', z_4'$ and the orientation reversing reflection of the surface about the plane through the odd curves $z_1', z_3'$.  
Of these $4$ maps only the identity map fixes both $x_{\{1,2,3\}}$ and $x_{\{2,3,4\}}$.  Since $f$ fixes both $x_{\{1,2,3\}}$ and $x_{\{2,3,4\}}$ but is not the identity it cannot be induced by an extended mapping class.
\end{proof}

\subsection{Theorems~\ref{thm:same} and \ref{thm:notiff} for sporadic surfaces of low complexity}
\label{SS:sporadic}
We address the ``missing cases'' in  Theorems~\ref{thm:same} and \ref{thm:notiff}.
The case  of $\surface^1_2$ was included in the statement of Theorem 2, but $\surface^0_2$ was omitted.  Indeed, as noted by Aramayona and Leininger \cite[$\S$3.1]{AL}, we have $X(\surface_2^0) = \mathfrak{X}(\surface_2^0)$ and hence $X(\surface_2^0)$ is a finite rigid set.

Recall that from the group cohomological point of view our nonstandard definition of the curve complexes $\CC(\surface_0^4)$, $\CC(\surface_1^0)$ and $\CC(\surface_1^1)$ as countably infinite discrete sets is appropriate. 
In particular, Propositions~\ref{prop:simpsphere} or  \ref{prop:geng} hold for these cases. 
Clearly these complexes have no finite rigid subsets.  In particular our $0$-spheres  $X(\surface_0^4)$, $X(\surface_1^0)$ and $X(\surface_1^1)$ cannot be finite rigid subsets of their respective curve complexes.

However, in these three sporadic cases if $\CC'$ is the standard 2-dimensional curve complex with $0$-skeleton given by the set of essential simple closed curves and a $k$-simplex for each set of $k+1$ curves with minimal pairwise intersection number, 
then our homologically non-trivial spheres in $\CC$ can be associated with rigid sets in $\CC'$.  Let $X'$ be the subcomplex of $\CC'$ spanned by $X$.
The finite rigid set $\mathfrak{X}$ of Aramayona and Leininger \cite[$\S$4]{AL} is a maximal 2-simplex and a simplicial injection of $\mathfrak{X}$ into $\CC'$ is induced by a unique extended mapping class.  
We may take our finite rigid set $X'$ to be a proper subcomplex of $\mathfrak{X}$.   A simplicial injection of $X'$ into $\CC'$ is induced by a unique mapping class.   In this case $X'$ is not a sphere, but a cell of dimension one.

\bigskip

\noindent
{\sc Joan Birman\\
Department of Mathematics, Barnard-Columbia\\
2990 Broadway\\
New York, NY 10027, USA\\}
\medskip
{\tt \verb|jb@math.columbia.edu|}

\noindent
{\sc Nathan Broaddus\\
Department of Mathematics\\
Ohio State University\\
231 W. 18th Ave.\\
Columbus, OH 43210-1174, USA\\}
\medskip
{\tt \verb|broaddus@math.osu.edu|}

\noindent
{\sc William W. Menasco\\
Department of Mathematics\\
University at Buffalo--SUNY\\
Buffalo, NY 14260-2900, USA\\}
\medskip
{\tt \verb|menasco@buffalo.edu|}
\end{document}